\documentclass{article}
\usepackage[utf8]{inputenc}
\usepackage{amsfonts}
\usepackage{amsmath,amsthm}
\usepackage{amscd}
\usepackage[all,cmtip]{xy}
\usepackage{amssymb}
\usepackage{geometry}
\usepackage{fancyhdr}
\usepackage{lipsum}
\usepackage{mathtools}
\usepackage{relsize}
\usepackage{physics}
\usepackage{braket}
\usepackage{bm}
\usepackage{esint}
\usepackage{enumitem}
\usepackage{mathrsfs}
\usepackage{xcolor}
\usepackage{bbold}
\usepackage{comment}
\usepackage[
backend=biber,
style=alphabetic,
maxbibnames=99
]{biblatex} %Imports biblatex package
\newtheorem{theorem}{Theorem}

\newtheorem{corollary}[theorem]{Corollary}

\newtheorem{definition}[theorem]{Definition}
\newtheorem{example}[theorem]{Example}

\newtheorem{lemma}[theorem]{Lemma}

\newtheorem{proposition}[theorem]{Proposition}
\newtheorem{remark}[theorem]{Remark}

\numberwithin{theorem}{section}

\def\text{\mbox}
\def\ra{\rightarrow}

\def\bi{\begin{itemize}}
\def\ei{\end{itemize}}
\def\bem{\begin{emuerate}}
\def\eem{\end{enumerate}}
\def\beq{\begin{eqnarray*}}
\def\eeq{\end{eqnarray*}}
\def\tb{\textbf}
\def\ti{\textit}

\def\N{\mathbb{N}} %naturals
\def\Z{\mathbb{Z}} %integers
\def\R{\mathbb{R}} %reals
\def\P{\mathbb{P}} %probability
\def\E{\mathbb{E}} %expectation
\def\T{\mathbb{T}} %Torus

\def\1{\mathbb{1}}
\def\actsl{\curvearrowright}

\newcommand*\Laplace{\mathop{}\!\mathbin\bigtriangleup}%Laplacian Operator
\title{Global Geometry within an SPDE Well-Posedness Problem}
\author{Hongyi Chen \\ hchen238@uic.edu \and Cheng Ouyang \\ couyang@uic.edu}
\begin{document}

\maketitle

\begin{abstract}
    \noindent On a closed Riemannian manifold, we construct a family of intrinsic Gaussian noises indexed by a regularity parameter $\alpha\geq0$ to study the well-posedness of the parabolic Anderson model. We show that with rough initial conditions, the equation is well-posed assuming non-positive curvature with a condition on $\alpha$ similar to that of Riesz kernel-correlated noise in Euclidean space.  Non-positive curvature was used to overcome a new difficulty introduced by non-uniqueness of geodesics in this setting, which required exploration of global geometry. The well-posedness argument also produces exponentially growing in time upper bounds for the moments. Using the Feynman-Kac formula for moments, we also obtain exponentially growing in time second moment lower bounds for our solutions with bounded initial condition.
\end{abstract}

\tableofcontents

%\section{Introduction}
\section{Introduction}
Let $M$ be a $d$-dimensional compact Riemannian manifold.
We consider the formal Cauchy problem 
\begin{align}\label{PAM}
    \begin{cases}
    \left(\partial_t+\frac{1}{2}\Laplace_M\right)u(t,x)=\beta u(t,x)\cdot {\dot{W}},\quad (t,x)\in \R_+\times M,\\
    u(0,x)=\mu,
    \end{cases}
\end{align}
where $\beta>0$ is a constant, $\Laplace_M=-\text{div}(\text{grad})$ is the Laplace-Beltrami operator (we follow the geometer convention with the negative sign), and $\mu$ is a finite measure on $M$. Equation \eqref{PAM} is usually named the parabolic Anderson model (PAM) in the literature. It arises in a large number of diverse questions in probability theory and mathematical physics. For example, it gives rise to the free energy of the directed polymer and to the Cole-Hopf solution of the KPZ equation \cite{ACQ,BertiniGiacomin97,CorwinKPZ,kardar87,KPZ86,MQR21}; it also has direct connections with the stochastic Burger's equation \cite{carmonaMolchanov,BCJ94} and Majda's model of shear-layer flow in turbulent diffusion \cite{majda93}.
We say that a random field $\{u(t,x)\}_{(t,x)\in \R_+\times M}$ is a mild solution to \eqref{PAM} if it satisfies  \begin{equation}\label{pam:mild}
    u(t,x)=\int_M P_t(x,y)\mu(dy)+\beta \int_0^t\int_M P_{t-s}(x,y)u(s,y)W(dy,ds)=J_0(t,x)+\beta I(t,x)
\end{equation}
where $P_t(x,y)$ is the heat kernel on $M$ and $J_0(t,x):=\int_M P_t(x,y)\mu(dy)$ is the homogeneous solution to the heat equation. 
The second integral $I(t,x)$ is to be understood in the sense of It\^{o}-Walsh \cite{Walsh86}.\\
%$W$ is a Gaussian noise defined on a probability space $(\Omega, \mathcal{F},\P)$ white in time and colored in space, as described below.

The It\^{o}-Walsh solution theory for \eqref{PAM} has been successful in $d=1$ with $W$ being space-time white noise since its introduction by Walsh. However, the white noise becomes too singular to apply the Walsh theory when space dimension is greater than 1; and solutions have been constructed in the Stratonovich sense via renormalization techniques (see \cite{HairerLabbe2D,HairerLabbe3D} for constant in time white noise on Euclidean spaces for $d=2,3$, and \cite{BailleulDangMouzard,HairerSingh,ShenZhuZhuPAM} for closed manifolds with constant in time white noise for $d=2,3$). Finer properties have been difficult to study in these scenarios due to the complex regularity structures associated to the system (exceptions being  \cite{GhosalYi,KönigPerkowskiZuijlen}). In particular, there is a lack of literature on the effects of geometry and topology on the properties of the solution.\\

In Euclidean space $\R^d$, \cite{Dalang98} extended the Walsh theory assuming bounded initial condition to $d\geq 2$ and noise white in time and colored in space with homogeneous covariance $G(x,x')=G(x-x')$. Here, the necessary and sufficient condition \begin{equation}\label{Dalang cond Rd}
    \int_{\R^d}\frac{\hat{G}(d\xi)}{1+\abs{\xi}^2}<+\infty,
\end{equation} was given for such noise, where $\hat{G}$ is the  Fourier transform of $G$. Condition \eqref{Dalang cond Rd} is usually referred as Dalang's condition in the literature and understood as a regularity requirement on $W$. The result in \cite{Dalang98} was extended to measure-valued initial conditions in \cite{ChenKim19}. In this setting, many interesting properties such as fluctuations \cite{CSZ17,DunlapGu22AOP,GCL23,KN24,Tao24}, spacial ergodicity\cite{CKNP21}, and intermittency \cite{CH19,ChenKim19}  were established.\\%While all of these results are interesting, the fact that well-posedness heavily relied on Fourier analysis (especially for measure-valued initial conditions) makes it difficult to study analogous problems on non-Euclidean spaces for noises with singularity. One notable extension for bounded initial conditions is \cite{BOTW-Heisenberg}, which showed that Dalang's condition can be affected by Sub-Riemannian degeneracy. However, this work also relied heavily on the good structure of the (group) Fourier transform on Sub-Riemannian Heisenberg groups. \cite{COV23} worked on $\T^d$, which showed that a type of Dalang's condition holds even for measure-valued initial conditions. It should be noted that despite $\T^d$ being a nice subset of Euclidean space, the presence of the Brownian bridge made it impossible to simply copy the argument from $\R^d$, but the geometric significance of this was hidden by Fourier analysis.\\

The present paper arises from the natural question of how the geometry of the underlying space (in our case a Riemannian manifold) would influence the behavior of the solution to equation \eqref{PAM}. Indeed, one expects that the underlying Brownian motion (associated to the Laplace operator) in the space plays an important role in the solution, whereas the movement of a Brownian motion certainly feels the geometry of the space. For example, it is well understood that Dalang's condition ensures the existence of a solution in the It\^{o} sense. When solving \eqref{PAM} on Heisenberg groups, it was shown in \cite{BOTW-Heisenberg} that Dalang's condition appears in the correct form in terms of the Hausdorff dimension instead of the topological dimension of the space. In addition, some new Lyapunov exponents have been revealed in \cite{BCHOTW} for solutions to \eqref{PAM} on metric measure spaces such as metric graphs and fractals. A recent paper \cite{BCO24} also studied \eqref{PAM} on Cartan-Hadamard manifolds. \\

All aforementioned three papers assumed that equation \eqref{PAM} starts from a nice initial condition, in which case solution theory largely relies on good heat kernel estimates. Keeping in mind the connection between the PAM and directed polymers, in this article we are interested in more general measure-valued initial conditions for \eqref{PAM}. Among these is the Dirac delta initial condition, for which solutions of \eqref{PAM} are already known to behave differently than those starting from nice initial conditions \cite{ACQ,CorwinKPZ,MQR21} in $d=1$. As an initial exploration of the connection between the geometry of a manifold and the equation \eqref{PAM}, we restrict our analysis on a compact Riemannian manifold throughout this paper. An immediate difficulty for the exploration in this direction is that one only starts to see interesting geometries of a manifold when the dimension is greater or equal to 2, however the white noise is already too singular to drive \eqref{PAM} in dimension 2 in the It\^{o} sense. To overcome this difficulty, we construct a family of intrinsic noises on manifolds that are white in time and colored in space. Moreover, they are more regular than space-time white noise in space variables so that one can still solve \eqref{PAM} in the framework of It\^{o}-Walsh. As one will see in Section 2 below, the spatial covariance function of our colored noise is a canonical function on the manifold and is the analogue of the Riesz kernel on $\mathbb{R}^d$. Similar constructions also appeared in \cite{COV23,BOTW-Heisenberg,BCHOTW, BCO24}. \\

The main result of this paper can be summarized as follows. More precise statements can be found in Theorem \ref{main result1} and Theorem \ref{lower bound} below.
\begin{theorem}\label{th:intro} Let $W=W_{\alpha,\rho},\alpha,\rho\geq0$ be a the noise given in Definition \ref{def: Colored Noise}.  Assume that Dalang's condition $\alpha>{(d-2)}/{2}$ holds and that $M$ is a compact Riemannian manifold with non-positive sectional curvature.\newline
(1)For any finite measure $\mu$ on $M$, the Cauchy problem (1.1) has a random field solution $\set{u(t,x)}_{t>0,x\in M}$ with the following exponential upper bound for some positive constants $C$ and $\theta$ depending on $p\geq2, \beta$ and $M$, 
    $$\E[|u(t,x)|^p]^{\frac{1}{p}}\leq CJ_0(t,x)e^{\theta t},\quad \mathrm{for\ all}\ t>0.$$
    Here $J_0(t,x)$ is the solution to the homogeneous heat equation starting from $\mu$.  
    
  (2) Suppose $\mu(dx)=f(x)dx$, where $f\in L^\infty(M)$ and $\inf_{x\in M}f(x)\geq \varepsilon>0$. Suppose in addition $\rho>0$. Then there exists a positive constant $c$ such that, $$\E[u(t,x)^2]\geq \varepsilon^2 e^{c\, t},\quad\mathrm{for\ all}\ t>0.$$ 
\end{theorem}

The exponential growth of moments in time has been linked to the study of intermittency, which is the presence of high peaks in the graph of the solution \cite{BertiniCancrini,carmonaMolchanov,khoshnevisan14,molchanov91}. For $d=1$ space-time white noise, \cite{BorodinCorwin14} gave a formula for the second moment starting from the Dirac delta initial condition using discrete approximations, which was re-proven in \cite{ChenThesis,ChenDalangAOP} using stochastic analysis, as a special case of a general result for measure-valued initial conditions. \cite{ChenThesis,ChenDalangAOP} also proved exponentially growing $p-$th moment upper bounds, which were extended to $d\geq 2$ in \cite{ChenKim19} with noise white in time and homogeneously colored in space assuming \eqref{Dalang cond Rd}. On compact manifolds, \cite{TindelViens} showed an exponentially growing in time almost sure upper bound hold for nice noise and uniform initial condition, which hints that intermittency is a local property. In \cite{BOTW-Heisenberg}, second moment upper and lower bounds were proven for bounded initial conditions on the sub-Riemannian Heisenberg group. For measure-valued initial conditions, $p-$th moment upper bounds and second moment lower bounds were shown in bounded domains in Euclidean space \cite{CCL23} and the Torus $\T^d$ \cite{COV23}, following the ideas of \cite{ChenThesis,ChenDalangAOP}. \\

Finally, let us briefly explain the main idea of our approach and where the curvature condition is used in order to prove Theorem \ref{th:intro}. We take the iteration procedure developed in \cite{ChenThesis,ChenDalangAOP,CJKS11} which study the PAM with measure-valued initial condition. An observation made in \cite{COV23} is that the success of their iteration procedure hinges on a careful analysis of the following integral,
\begin{align}\label{main quantity for iteratin}\int_0^t ds\int_{M^2}dzdz' P_{t,x_0,x}(s,z)P_{t,x_0',x'}(s,z')\bm{G}_{\alpha,\rho}(z,z'),\end{align}
where $P_{t,x,y}(s,z)$ is the density of the Brownian bridge that starts at $x$ and reaches $y$ at time $t$, and $\bm{G}_{\alpha,\rho}$ is the spatial covariance function of the noise.
Clearly a proper estimate of the above integral requires a good understanding of how the measure of a Brownian bridge is concentrated for all time $t$ and $0<s<t$ and for all $x$ and $y$. Since the density of a Brownian bridge can be expressed in terms of the heat kernel,
$$P_{t,x,y}(s,z)=\frac{P_{s}(x,z)P_{t-s}(z,y)}{P_t(x,y)},$$
and one usually expects a Gaussian type heat kernel estimate (see  Lemma \ref{th: heat kernel upper bound} below), the concentration of the measure of a Brownian bridge is thus controlled by the interplay of three distance functions (coming from the exponential terms in the heat kernel estimates),
  \begin{align}\label{func: F}F_{s,t;x,y}(z):=-\frac{\bm{d}(x,y)^2}{2t}+\frac{\bm{d}(x,z)^2}{2s}+\frac{\bm{d}(z,y)^2}{2(t-s)}.\end{align}

This is where global geometry enters and imposes the main difficulty for our analysis. Indeed, it is not hard to see that $F$ takes its minimum when $z$ lies on geodesics connecting $x$ and $y$. Hence, most of the measure is concentrated around the minimizer on the geodesic, especially for small $t$.  When $x$ and $y$ are in the cut-locus of each other, there are multiple (possibly infinitely many) distance minimizing geodesics connecting $x$ and $y$ making the analysis of $F$ not easily accessible. In order to tackle this difficulty, we assume throughout our discussion that the sectional curvature of $M$ is non-positive. This curvature condition ensures that there are only finitely many distance minimizing geodesics between an two points, which simplifies the analysis; it also allows to give a careful analysis of $F$ by comparing triangles on $M$ to Euclidean triangles. To our best knowledge, this is the first instance of global geometry appearing in the study of well-posedness for a linear differential equation of this type. The exponential lower bound of the second moment of the solution stated in Theorem \ref{th:intro}-(2) is due to the compactness of the manifold $M$; thus the Brownian motion is ergodic. The extra assumption of a nice initial condition allows us to use the Feynman-Kac formula for the second moment.\\

It is not clear at the moment whether the non-positive curvature condition assumed in Theorem \ref{th:intro} is only a technical condition or not. However, from the analysis below, we expect that Dalang's condition might take a different form when $M$ is a sphere given that the measure of Brownian bridge concentrates around lines of latitude (as opposed to finite many points under the non-positive curvature condition) when $t$ is small and when $x$ and $y$ are antipodal points. This will be investigated in a subsequent work. In addition, we think our approach is quite general and is robust enough to be extended to study the PAM with measure-valued initial data in other complex spaces, such as fractals. \\

The rest of the paper is organized as follows. In Section 2, we construct a family of colored noises on $M$ that are smoother than the white noise. In Section 3, we introduce the iteration procedure developed \cite{CJKS11,ChenThesis,ChenDalangAOP} and analyze the integral in \eqref{main quantity for iteratin}. Along the way, we identify the specific geometric difficulty mentioned above: the analysis of the function $F$ given in \eqref{func: F}. All is then related to the geometry of geodesics because the Brownian bridge in short time sees the number of minimizing geodesics (see \cite{hsu90} for a large deviation characterization of this statement), and the fact that they are finite for non-positively curved manifolds gives us a handle on it. The execution of this intuition is laid out in Section 3.2. It requires precise use of the geometry and topology of negatively curved spaces and is the core of the paper. Once all the estimates needed for the iteration are in place, the well-posedness and moment bound of the solution follow similarly to the Euclidean case, which is the content of Section 4.  Finally in Section 5, we use Feynman-Kac formula for moments and the structure of our noise to produce a lower bound which grows exponentially in time, which strengthens the belief that the solution is intermittent on all compact manifolds.

\bigskip
%Finally, we convention of notation and constants $C$.  Balls in manifold $B(x,\epsilon)$ and on $\mathbb{R}^n$ by $B_{\mathbb{R}^n}(x,\epsilon).$

We list here some conventions and notations we employ in the rest of the paper.
\begin{itemize}
    \item  We follow convention and use $C_1,C_2,C_3$  and $c_1,c_2$ etc. to denote generic constants that are independent of quantities of interest. We will also use $C_M$ to denote a constant depending on $M$. The exact values of these constants may change from line to line.
    \item For $x\in M$ and $r>0$, $B(x,r)$ will denote the geodesic ball of radius $r$ centered at $x$.
    \item $B_{\R^d}(r)$ will be a ball of radius $r>0$ in $\R^d$.
    \item $m_0=\int_M dx$ will be the volume of the manifold.
    \item $i_M>0$ will be the injectivity radius of $M$. We will also fix a constant $\delta=i_M/8$.
    \item $\bm{d}(x,y)$ will denote the distance between $x,y\in M$.
\end{itemize}

\section{Colored Noise on Compact Riemannian Manifolds}

In order to construct a (centered) Gaussian noise on $M$ smoother than the white noise, one essentially needs a positive-definite function $\mathbf{G}(x,y)$ on $M\times M$ that is less singular than the Dirac delta on diagonal $D=\{(x,x); x\in M\}$. When $M=\mathbb{R}^d$, such functions can be obtained through Fourier transforms, thanks to Bochner's theorem. On a compact manifold $M$, the spectral decomposition of the Laplace-Beltrami operator (which corresponds to the "Fourier transform" on $M$) becomes handy. In this section, we construct an intrinsic family of Gaussian noises on $M$ that we call \textit{colored noise} on manifold. As we will see below, these noises are smoother than space-time white noise and allows us to study \eqref{PAM} in the It\^{o} sense.\\

Denote by $0=\lambda_0<\lambda_1\leq\lambda_2\leq\dots$ the eigenvalues of $\triangle_M$ and by $\phi_0, \phi_1,\phi_2,\dots$ an orthonormal sequence of corresponding eigenfunctions. Thus $\triangle_M \phi_n=\lambda_n \phi_n$ and $\int_M \phi_i\phi_j dm=\delta_{ij}$. 
For any $\varphi\in L^2(M)$, there
is a unique decomposition
\begin{align}\label{E: func decomp}
  \varphi(x) = \sum_{n\geq0}a_n\phi_n(x).
\end{align}
In particular, $a_0={m_0}^{-1/2}\int_M\varphi  dm$ where $m_0=m(M)$ is the volume of $M$.

We introduce a family of spatial Gaussian noises $\dot{W}$ on $M$ with parameters
$\alpha$, $\rho\ge 0$ as follows. Let $(\Omega, \mathcal{F},\mathbb{P})$ be a
complete probability space such that for any $\varphi$ and $\psi$ on
$M$ both ${W}\left(\varphi\right)$ and
${W}\left(\psi\right)$ are centered Gaussian random variables with
covariance given by
\begin{align}\label{E:NoiseCov}
 \mathbb{E} \left({W}\left(\varphi\right){W}\left(\psi\right)\right)
    = \langle\varphi,\psi\rangle_{\alpha,\rho} :=\rho a_0 {b}_0+\sum_{n\not=0} \frac{a_n{b}_n}{\lambda_n^{\alpha}}
   \end{align}
where $a_n$'s and $b_n$'s are the coefficients of $\varphi$ and $\psi$ in decomposition \eqref{E: func decomp},
respectively. For $\rho>0$, let $\mathcal{H}^{\alpha,\rho}$ be the completion of
$L^2(M)$ under $\langle\cdot,\cdot\rangle_{\alpha,\rho}$. It is clear that $\mathcal{H}^{\alpha,\rho}$ is a Hilbert space, and, by general construction (see, e.g., \cite[Chapter 1.1]{nualart2006malliavin}), 
one obtains an abstract Wiener space $(\Omega, \mathcal{H}^{\alpha,\rho},\mathbb{P})$.

\begin{remark}When $\rho=0$, some special care is needed in order to identify a suitable
Hilbert space $\mathcal{H}^{\alpha,0}$. Let $L^2_0(M)$ be the space of $L^2(M)$
functions on $M$ such that $a_0=0$. Denote by
$\mathcal{H}^{\alpha}_0$ the completion of $L^2_0$ under
$\langle\cdot,\cdot\rangle_{\alpha,\rho}$. One could have set
$\mathcal{H}^{\alpha,0}=\mathcal{H}^{\alpha}_0$. However, when solving SPDEs on
compact manifolds, it is desirable to consider Wiener integrals
${W}\left(\varphi\right)$ where $\varphi$ is a function on the
manifold such that $a_0=\frac{1}{m_0}\int_{M}\varphi(x) d x \not=0$, where $m_0$ is the volume of $M$.
For this purpose, consider $\mathcal{H}^{\alpha}_0+\mathbb{R} \coloneqq
\{\varphi+c: \varphi\in\mathcal{H}^{\alpha}_0, \textnormal{and}\
c\in\mathbb{R}\}$. We can identify $\mathcal{H}^{\alpha}_0+\mathbb{R}$ with
$\mathcal{H}^{\alpha}_0$ through the equivalence relation $\sim$, in which
$\varphi\sim\psi$ if $\varphi-\psi$ is a constant. Finally, we set
\begin{align*}
  \mathcal{H}^{\alpha,0}=(\mathcal{H}^{\alpha}_0+\mathbb{R})/ \sim.
\end{align*}
%With this construction, it is clear that
%$\dot{W}\left(1_{[0,t]}\varphi\right)=\dot{W}\left(1_{[0,t]}(\varphi+c)\right)$
%for any $\varphi\in\mathcal{H}^{\alpha,0}$ and $c\in\mathbb{R}$. 
Throughout the
rest of our discussion, we will also adopt the short-hand $\mathcal{H}^\alpha$
for $\mathcal{H}^{\alpha,0}$.
\end{remark}

\begin{remark}
  It is clear from~\eqref{E:NoiseCov} that $L^2(M)\subset
  \mathcal{H}^{\alpha,\rho}\subset\mathcal{H}^{\beta,\rho}$ for
  $0\leq\alpha<\beta$. Moreover, the colored noise includes the white noise on $M$ if we pick $\rho=1$ and $\alpha=0$.
\end{remark}

The covariance structure $\langle\cdot,\cdot\rangle_{\alpha,\rho}$ admits a kernel. Indeed, let $p_t(x,y)$ be the heat kernel on $M$ and set for $\alpha,\rho>0$, 
\begin{align}\label{def: G_alpha etc}\bm{G}_\alpha(x,y):=\frac{1}{\Gamma(\alpha)}\int_0^\infty t^{\alpha-1}\left(P_t(x,y)-\frac{1}{m_0}\right) dt,\quad \textnormal{and}\ \ \bm{G}_{\alpha,\rho}(x,y):=\frac{\rho}{m_0}+\bm{G}_\alpha(x,y).
\end{align}

%{\color{red}It is easy to see(still needs address Referee 1 comment here)} that one has

By the spectral representation of the heat kernel 
$$P_t(x,y)=\frac{1}{m_0}+\sum_{n\geq 1}e^{-\lambda_n t}\phi_n(x)\phi_n(y),$$
one has
\begin{align}\label{covar spectr}
\bm{G}_\alpha(x,y)=\sum_{n\geq 1}\frac{1}{\lambda_n^\alpha}\phi_n(x)\phi_n(y),
\end{align}hence
\begin{align*}
\langle\varphi,\psi\rangle_{\alpha,\rho}=\int_{M^2}\phi(x)\bm{G}_{\alpha,\rho}(x,y)\psi(y)m(dx)m(dy).
\end{align*}

%{\color{blue}I think we don't really need this remark anymore?}
\begin{remark}
    It is clear from \eqref{covar spectr} that $\bm{G}_\alpha$ is the analogue of the Riesz kernel on $\mathbb{R}^d$. By \eqref{def: G_alpha etc} one has $\int_M \bm{G}_\alpha(x,y)m(dy)=0$. Hence $\bm{G}_\alpha$ is not non-negative. However, it can be shown that $\bm{G}_\alpha$ is bounded below on $M$ (see \cite{COV23} for example). We therefore can always pick a large enough $\rho$ so that the spatial covariance function $\bm{G}_{\alpha,\rho}$ is non-negative. 
\end{remark}

The following proposition gives the regularity of $\bm{G}_{\alpha}$ (hence $\bm{G}_{\alpha,\rho}$ as well) on diagonal. 

\begin{proposition}\label{Prop: G_alpha}
    For any $\alpha> 0$, we have $$\abs{\bm{G}_{\alpha}(x,y)}\leq \begin{cases}
        C_\alpha,& \alpha>d/2\\
        C_\alpha(1+\log^-\bm{d}(x,y)),& \alpha=d/2\\
        C_\alpha \bm{d}(x,y)^{2\alpha-d},& \alpha<d/2.
    \end{cases}$$
    Where $\log^-(z)=\max(z,-\log z)$ and $\bm{d}(x,y)$ is the Riemannian distance on $M$.
\end{proposition}
\begin{proof}
    See \cite{Brosamler}.
\end{proof}

%{\color{red}Referee 1 wants us to make a remark here}\\

%To close the discussion in this section, we define the noise on $\R_+\times M$ that is white in time and colored in space. 

Thanks to Proposition \ref{Prop: G_alpha}, the colored noise constructed above is indeed smoother than white noise for all $\alpha>0$, and defines a worthy martingale measure in the sense of Walsh\cite{Walsh86}.

\begin{definition}\label{def: Colored Noise}
    Let $\alpha>0$ and consider the following Hilbert space of space-time functions,
\begin{align}\label{eq-hilb}
\mathcal{H}_{\alpha,\rho}=L^2(\mathbb{R}_+, \mathcal{H}^{\alpha,\rho}).  
\end{align}
On a complete probability space $(\Omega, \mathcal{F},\mathbb{P})$ we define a centered Gaussian family $\{W_{\alpha,\rho}(\phi); \phi\in L^2(\R_+)\cap\mathcal{H}_{\alpha,\rho}(M)\}$, whose covariance is given by 
\[
\mathbf E\left[ W_{\alpha,\rho}(\varphi) W_{\alpha,\rho}(\psi)\right]
=\int_{\R_+}\ \left\langle \varphi (t,\cdot) , \psi (t,\cdot)\right\rangle_{\alpha,\rho}  dt \, ,
\]
for $\varphi$, $\psi$ in $\mathcal{H}_{\alpha,\rho}$ in the space variable. This family is called
colored noise on $M$ that is white in time.
\end{definition}

To simplify notation, we will drop the indexes $\alpha$ and $\rho$ and use $W$ for $W_{\alpha,\rho}$ throughout the rest of the paper. We will also write $dz$ instead of $m(dz)$ when integrating over $M$.

\section{The $\rhd$ operator and $\mathcal{L}_n$}
In order to establish the existence and uniqueness of the solution to equation \eqref{PAM} with measure-valued initial condition, we follow the iteration strategy developed in \cite{ChenThesis,ChenDalangAOP}. For this purpose, we introduce: 
\begin{comment}
    {\color{blue} (this would be [Chen, Dalang '15] Annal of Probability, right?)}. {\color{red} The AOP is basically the paper version of his thesis. It is better to refer back to the very original work if possible.}

{\color{red}Going forward, we fix $\rho>0$ such that $\bm{G}_{\alpha,\rho}>0$.
The following is adapted from [Chen, Kim].  (Do we need $\bm{G}_{\alpha,\rho}>0$?) }{\color{blue} It appears we don't for well-posedness, we still need $\rho>0$ for the lower bound, though. We might need $\bm{G}_\alpha,\rho>0$ for positivity in a future project?}
\end{comment}

\begin{definition}
   Let $M^4$ be the Cartesian products of four copies of $M$. For $h,w:\R_+\times M^4\ra \R$,  define the  operator $\rhd$ by $$h\rhd w(t,x_0,x,x_0',x'):=\int_0^t ds \iint_{M\times M}dzdz' h(t-s,z,x,z',x')w(s,x_0,z,x_0',z')\bm{G}_{\alpha,\rho}(z,z').$$
Define $\set{\mathcal{L}_n}_{n\geq 0}$ recursively by \begin{align}\label{def: L_n}
    \mathcal{L}_n(t,x_0,x,x_0',x'):=\begin{cases}
        P_t(x_0,x)P_t(x_0',x'),&n=0\\
        \mathcal{L}_0\rhd \mathcal{L}_{n-1}(t,x_0,x,x_0',x'),&n>0.
    \end{cases}
    \end{align}
\end{definition}
%\begin{remark}
%     If $h,w$ are symmetric in $(x_0,x)$ and $(x_0',x')$, this operator is obviously associative. It is however not commutative in general.  {\color{red}(double check this remark.  also, why do we need it?) } {\color{blue}We don't need it, just delete it in the next revision.}
%\end{remark}

The role played by $\mathcal{L}_n$ can be formally explained as follows.  Recall $J_0(t,x)=\int_M P_t(x,y)\mu(dy)$ is the solution to the homogeneous heat equation starting from $\mu$, and define%{\color{red} $J_0$ needs to be introduced} {\color{blue}$J_0$ was introduced in section 1 (if it's better to do so, we can move it here).} {\color{red}Where in Section 1?  If it was introduced before, we can recall the definition of $J_0$ here.
$$J_1(t,x,x'):=J_0(t,x)J_0(t,x'),\quad  g(t,x,x'):=\E[u(t,x)u(t,x')].$$
It\^{o} isometry then implies % {\color{blue}Three equations below are the same.}
    $$g(t,x,x')=J_1(t,x,x')+
    \beta^2\int_0^t ds \iint_{M^2}dzdz'P_{t-s}(x,z)P_{t-s}(x',z')\bm{G}_{\alpha,\rho}(z,z')g(s,z,z').$$
   % g(s,z,z')&=J_1(s,z,z')+\beta^2\int_0^s ds_1 \iint_{M^2}dz_1dz_1' P_{s-s_1}(z,z_1)P_{s-s_1}(z',z_1')\bm{G}_{\alpha,\rho}(z_1,z_1')g(s_1,z_1,z_1'),\\
    %g(s_1,z_1,z_1')&=J_1(s_1,z_1,z_1')+\beta^2\int_0^{s_1} ds_2 \iint_{M^2}dz_2dz_2' P_{s_1-s_2}(z_1,z_2)P_{s_1-s_2}(z_1',z_2')\bm{G}_{\alpha,\rho}(z_2,z_2')g(s_2,z_2,z_2').

Iterating the above relation suggests the following formal equality:
\begin{align}
    g(t,x,x')=J_1(t,x,x')+&\sum_{n=0}^\infty \beta^{2n+2}\int_{0\leq s_n\leq s_{n-1}\leq \cdots \leq s_0\leq t}ds_n\cdots ds_0\iint_{M^{2n+2}}dz_0dz_0'\cdots dz_ndz_n'\nonumber\\
    &\quad\times J_1(s_n,z_n,z_n')\prod_{k=0}^n P_{s_{k-1}-s_k}(z_{k-1},z_k)P_{s_{k-1}-s_k}(z_{k-1}',z_k') \bm{G}_{\alpha,\rho}(z_k,z_k').
\end{align}
Writing $$J_1(s_n,z_n,z_n')=\int_{M^2}\mu(dz)\mu(dz') P_{s_n}(z_n,z)P_{s_n}(z_n',z'),$$ we have
\begin{align}\label{iteration corelation}g(t,x,x')=J_1(t,x,x')+\beta^2\iint_{M^2} \mu(dz)\mu(dz') \sum_{n=0}^\infty \beta^{2n}\mathcal{L}_n(t,x,z,x',z').\end{align}
Observe that the validity of the above computation relies on convergence of the following series: \begin{equation}\label{K_beta}
    \mathcal{K}_\beta(t,x,z,x',z'):=\sum_{n=0}^\infty \beta^{2n}\mathcal{L}_n(t,x,z,x',z').
\end{equation}
It has been shown in \cite{ChenThesis,ChenDalangAOP} that the existence and uniqueness of a solution to equation \eqref{PAM} as well as moment estimates of the solution hinge on proper estimates of $\mathcal{L}_n$. It also has been shown in the same papers that $\mathcal{L}_n$ can be controlled inductively by a proper estimate of $\mathcal{L}_1$. The rest of this section is thus devoted to the analysis of $\mathcal{L}_n, n\geq1$. More specifically, we obtain estimates of $\mathcal{L}_1$ in Sections 3.1 and 3.2 for large and small time, respectively. Then an iteration procedure gives the estimate of $\mathcal{L}_n, n\geq2$ in Section 3.3. Once the estimates of $\mathcal{L}_n$ are in place, we will address the well-posedness and moment bounds of equation \eqref{PAM} in Section 4.

\begin{remark}
    As an alternative to the iteration method described above, one can apply the method in \cite{Huang2016OnSH}, that is,  by dividing both sides of \eqref{pam:mild} by $J_0(t,x)$ and considering the norm
    $\sup_{x\in M}\left\|\frac{u(t,x)}{J_0(t,x)}\right\|_p.$
    However, the heart of the problem is still the estimate of \eqref{main quantity for iteratin} (or equivalently $\mathcal{L}_1$), so all of the geometric machinery in Section 3 remains necessary. 
\end{remark}

% We will need the following. {\color{blue}I think Elton pointed out that the upper bound below requires a non-positive (sectional) curvature conditions, correct?}{\color{red} Yes, but I can't find the exact reference for it. Should I ask him or will you do it?} {\color{blue}I think you said last semester you found the reference in Elton's book, otherwise, we can try to find it later;[''''''''''''''''''''.  Also, we need to point out the sectional curvature condition in the Lemma.}\\
%{\color{red}The lemma has been fixed. I believe you misheard, I meant to say I checked both his book and the references he gave in the books and was unable to find a statement which we can cite.}

\bigskip
We first recall the following heat kernel upper bound on a non-positively curved compact Riemannian manifold.\\% (reference here \cite{}). 
%{\color{red} Reviewer 4 says we can simply cite \cite{LiYau86}.}{\color{blue}Can you check to see if LY86 contains the statement we need?}\\
%{\color{red}The main theorem of \cite{LiYau86} says this: if $\mathrm{Ric}_M\geq -(d-1)K$, then for any $0<\epsilon\leq1$ and $1<a\leq2$ we have $$P_t(x,y)\leq C_{d,\epsilon,a}m(B(x,\sqrt{t}))^{-\frac{1}{2}}m(B(y,\sqrt{t}))^{-\frac{1}{2}}\exp\left(-\frac{\bm{d}(x,y)^2}{(2+\epsilon)t}+\frac{C_d}{a-1}K\epsilon t\right).$$
%Since $m(B(x,R))\sim R^{d}$ for small $R$ on all manifolds for $R<i_M$, this should be sufficient.}
\begin{lemma}\label{th: heat kernel upper bound}  Let $M$ be compact with non-positive sectional curvature. For any $m\geq1$, we have 
\begin{align}\label{def: G_t}P_t(x,y)\leq (2\pi t)^{-\frac{d}{2}}\exp\left(-\frac{\bm{d}(x,y)^2}{2t}\right)+C_H(t^m\wedge 1),\end{align}
for all $t>0,x,y\in M$ and some $C_H>0$.
\end{lemma}
\begin{proof}  For large $t$, \eqref{def: G_t} follows from the following standard estimate\cite[Chapter 3]{JostRG} on compact manifolds: there exist $\alpha>0, C>0$ such that
\[\sup_{x,y\in M}|P_t(x,y)-m_0^{-1}|\leq Ce^{-\alpha t},\quad t\geq1.\]

The curvature condition is used for small $t$, under which there are only finite many distance minimizing geodesics connecting any two point $x,y\in M$.   The discussion in the proof of Theorem 5.3.4 of \cite{hsu02} therefore implies in short time (say $0<t<1$) we have \[P_t(x,y)\leq \frac{C}{t^{d/2}}e^{-\frac{\bm{d}(x,y)^2}{2t}}.\]   Combining these finishes the proof.
\end{proof}
%{\color{red}To prevent distractions, we make the following remark for a reader unfamiliar with the heat kernel literature.}
\begin{remark}
%Recall that $i_M$ is the injectivity radius of $M$. It is well-known that for $x,y\in M$ with $\bm{d}(x,y)\leq i_M$, one has (see, e.g., \cite[Theorem 5.1.1]{hsu02} or \cite{LiYau86})
%\begin{align}\label{HK 1}
%P_t(x,y)\leq \frac{C}{t^{d/2}}e^{-\frac{\bm{d}(x,y)^2}{2t}},\quad 0<t<1.
%\end{align}
%When $\bm{d}(x,y)> i_M$, the following holds on any compact manifold $M$ of dimension $d$ (Theorem 5.3.4 of \cite{hsu02}),
%\begin{align}\label{HK 2}P_t(x,y)\leq \frac{C}{t^{(2d-1)/2}}e^{-\frac{\bm{d}(x,y)^2}{2t}},\quad 0<t<1. \end{align}
%Clearly \eqref{HK 1} and \eqref{HK 2} imply the global bound,
The main result of \cite{LiYau86} implies that %one has
%\begin{align}\label{relaxed HK bound}
%P_t(x,y)\leq \frac{C}{t^{d/2}}e^{-\frac{\theta \bm{d}(x,y)^2}{2t}},\quad  x,y\in M;\  0<t<1,
%\end{align}
for any fixed $\theta\in (0,1)$, one has
%We could have used \eqref{relaxed HK bound} in the proof of Lemma \ref{th: heat kernel upper bound} for small $t$ and obtained the slightly non-sharp bound
\begin{align}\label{relaxed HK bound 2}
   P_t(x,y) \leq (2\pi t)^{-\frac{d}{2}}\exp\left(-\frac{\theta \bm{d}(x,y)^2}{2t}\right)+C_H(t^m\wedge 1).
\end{align}

As one will see below, \eqref{relaxed HK bound 2} is sufficient for our analysis. For convenience, we proved the optimal bound with $\theta=1$, for which finitely many geodesics is necessary.
\end{remark}

To proceed, we make a remark on some elementary computations that will be used repeatedly in the sequel.

\begin{remark}\label{rk: turn to eclidean integral}
Throughout the paper, we denote the injectivity radius of $M$ by $i_M$.  Note that for $\delta=i_M/8$ one has %{\color{blue}We need $\alpha>0$ below, correct?}{\color{red} Yes, in the previous section everything is only defined for $\alpha>0$.}
\begin{align}\label{L1 bound for covar}\norm{\bm{d}(z,\cdot)^{2\alpha-d}}_{L^1(M)}&=\left(\int_{B(z,\delta)}+\int_{B(z,\delta)^c}\right)dz' \bm{d}(z,z')^{2\alpha-d}\nonumber\\
&\leq C_M \int_{B_{\R^d}(0,\delta)} \abs{x}^{2\alpha-d}dx+\frac{m_0}{\delta^{d-2\alpha}}=c_{\alpha,M}.\end{align}
The above estimate will be used repeatedly to bound $\norm{\bm{d}(z,\cdot)^{2\alpha-d}}_{L^1(M)}$ in the sequel. The inequality in \eqref{L1 bound for covar} follows by taking the integral into geodesic normal coordinates around $z$. The estimate is uniform in $z$ thanks to the compactness of $M$. %that gives a uniform upper bound for the determinant of the Jacobian (dependent on curvature lower bound).%, which can be determined from the metric tensor.
    This procedure will be performed every time when moving an integral into normal coordinates without stating so in the rest of the paper. In particular, the computation below will be used repeatedly later:
 \begin{align}\label{int Gaussian M}
      \int_M dz (2\pi s)^{-d/2}e^{-\frac{\bm{d}(z,x)^2}{2s}}
       & =\left(\int_{B(x,\delta)}+\int_{B(x,\delta)^c}\right)dz (2\pi s)^{-d/2}e^{-\frac{\bm{d}(z,x)^2}{2s}}\nonumber\\
&\leq  C_M\int_{B_{\mathbb{R}^d}(0,\delta)} (2\pi s)^{-d/2}e^{-\frac{\abs{z}^2}{2s}}dz +m_0(2\pi s)^{-d/2}e^{-\frac{\delta^2}{2s}}\nonumber\\
  &\leq c_M,\quad\mathrm{for \ all}\  s\geq0.     
       %& \lesssim 1+(2\pi s)^{-d/2}e^{-\frac{\delta^2}{2s}}=O(t),\quad\mathrm{as}\ t\uparrow\infty.
    \end{align}
    
\end{remark}

Now we focus on obtaining a proper upper bound of $\mathcal{L}_1$. For simplicity, throughout our discussion below, we will take $m=1$ in \eqref{def: G_t}, and set
\begin{align}\label{def: G_t(x,y)}
G_t(x,y):=(2\pi t)^{-\frac{d}{2}}\exp\left(-\frac{\bm{d}(x,y)^2}{2t}\right)+C_H(t\wedge 1),
\end{align}
%{\color{red}(Referee 1 says give above $C$ a name)} and
\begin{align}\label{def: G_{t,x,y}(s,z)}
G_{t,x,y}(s,z):=\frac{G_{t-s}(x,z)G_s(z,y)}{G_t(x,y)}.
\end{align}

Lemma \ref{th: heat kernel upper bound} implies \begin{align}\label{bound: p by G}
    P_{t-s}(x_0,z)P_s(z,x)\leq&G_t(x_0,x)\frac{G_{t-s}(x_0,z)G_s(z,x)}{G_t(x_0,x)}=G_t(x_0,x)G_{t,x_0,x}(s,z).
\end{align}
Recall the definition of $\mathcal{L}_n$ in \eqref{def: L_n}, in particular
$$\mathcal{L}_1(t,x_0,x,x_0',x')=\int_0^t ds\int_{M^2}dzdz' P_{t-s}(x_0,z)P_s(z,x)P_{t-s}(x_0',z')P_s(z',x')\bm{G}_{\alpha,\rho}(z,z').$$
Thus \eqref{bound: p by G} gives,
$$\mathcal{L}_1(t,x_0,x,x_0',x')\leq G_t(x_0,x)G_t(x_0',x')\int_0^t ds\int_{M^2}dzdz' G_{t,x_0,x}(s,z)G_{t,x_0',x'}(s,z')\bm{G}_{\alpha,\rho}(z,z').$$
Deviating from existing literature \cite{ChenThesis,ChenDalangAOP,ChenKim19,COV23}, the analysis of \eqref{main quantity for iteratin} will be replaced with that of 
\begin{align}\label{main quantity for interation G}\int_0^t ds\int_{M^2}dzdz' G_{t,x_0,x}(s,z)G_{t,x_0',x'}(s,z')\bm{G}_{\alpha,\rho}(z,z').\end{align}
The reason we switch from \eqref{main quantity for iteratin} to \eqref{main quantity for interation G} (that is, swithcing from $P_{t,x,x_0}(s,z)$ to $G_{t,x,x_0}(s,z)$ ) is that $G_{t,x,y}(s,z)$ takes a rather explicit form and still captures the main property of $P_{t,x,y}(s,z)$; however, a good estimate of $P_{t,x,x_0}(s,z)$ may require both  heat kernel upper bound and lower bound.  

\bigskip
An upper bound of \eqref{main quantity for interation G} will be obtained by dividing the cases according to $t\geq \varepsilon$ and $t<\varepsilon$ for a prefixed small $\varepsilon>0$.

\subsection{Upper bound of $\mathcal{L}_1$ for large time ($t\geq \varepsilon$)}
The following upper bound of $\mathcal{L}_1$ is the main result of this section. It relies on the  observation that $G_{t,x,y}(s,z)$ is comparable to $G_s(x,z)$ when $t$ is large and $s<t/2$, which will be detailed in \eqref{C_epsilon} below. In this case, computations are local and do not depend on the global geometry of $M$.
\begin{theorem}\label{theorem:large time L1 bound}
    Assume $\frac{d}{2}>\alpha>\frac{d-2}{2}$ and fix $\varepsilon>0$. Recall the definition of $G_t(x,y)$ in \eqref{def: G_t(x,y)} and
    set $$k_1(s):=\sup_{x,x'\in M}\int_{M^2} dzdz' G_s(x,z)G_s(x',z')\bm{d}(z,z')^{2\alpha-d},\quad s>0.$$
We have, 
$$k_1(s)\leq C_{\alpha,M}(1+s^{\frac{2\alpha-d}{2}}),$$
    for some positive constant $C_{\alpha,M}$ depending on $\alpha$ and $M$. Moreover, for all $t\geq\varepsilon,$ 
    \begin{align}\label{bound:L1 by k1}\mathcal{L}_1(t,x_0,x,x_0',x')\leq C_L G_t(x_0,x)G_t(x_0',x')\left(\int_0^t k_1(s)ds\right),\end{align}
    where $C_L$ is a positive constant depending on $\varepsilon$ and $M$. % only on $\alpha,d,\delta=i_M/8,\varepsilon$ and $M$.      
    
    %{\color{red} the  curvature lower bound  (I guess the subscript $L$ in $C_L$ represents the curvature lower bound? Need to be more specific about what you mean by curvature lower bound here}. {\color{blue}The subscript $L$ is supposed to denote large time. In small geodesic balls the volume form can be bounded by the Ricci curvature lower bound, which is mentioned later (we may have to state this more clearly)}

\end{theorem}
\begin{proof}
    Recall that $$\mathcal{L}_1(t,x_0,x,x_0',x')=\int_0^t ds\int_{M^2}dzdz' P_{t-s}(x_0,z)P_s(z,x)P_{t-s}(x_0',z')P_s(z',x')\bm{G}_{\alpha,\rho}(z,z').$$
    Write the time integral $\int_0^t=\int_0^{\frac{t}{2}}+\int_{\frac{t}{2}}^t$. %By the symmetry of $s$ and $t-s$ in $G_{t,x_0,x}(s,z)$ and the fact that $M$ is compact, 
    We first consider $\int_0^{\frac{t}{2}}$ and show that for some positive constant $C_1$ depending on $\epsilon, M$ (but not on $t$, $x_0,x, x_0'$ and $x'$), one has
\begin{align}\label{est t/2}
\int_0^\frac{t}{2}\leq C_1G_t(x_0,x)G_t(x_0',x')\left(\int_0^\frac{t}{2}k_1(s)ds\right),\quad\mathrm{for\ all}\ \ t\geq \varepsilon.
\end{align}
Then by the symmetry of $s$ and $t-s$ in the definition of $\mathcal{L}_1$, a change of variables $s'=t-s$ gives the same bound for $\int_{t/2}^t$, that is
\begin{align*}
\int_\frac{t}{2}^t\leq C_1G_t(x_0,x)G_t(x_0',x')\left(\int_0^\frac{t}{2}k_1(s)ds\right).
\end{align*}
We thus conclude, observing the positivity of $k_1(s)$, that for all $t\geq\varepsilon$
\begin{align}\mathcal{L}_1%&=\int_0^\frac{t}{2}+\int_\frac{t}{2}^t\nonumber \\
&\leq 2C_1G_t(x_0,x)G_t(x_0',x')\left(\int_0^\frac{t}{2}k_1(s)ds\right)\nonumber\\
&\leq C_LG_t(x_0,x)G_t(x_0',x')\left(\int_0^tk_1(s)ds\right),
\end{align}
which gives the desired upper bound \eqref{bound:L1 by k1}.   To finish the proof, we need to establish \eqref{est t/2} and 
\begin{align}\label{est: int k1}
k_1(s)\leq C_{\alpha,M}(1+s^{\frac{2\alpha-d}{2}}),\quad \mathrm{for\ all}\ s>0.
\end{align}
To this aim, set
$$\tilde{\mathcal{L}}_1(t,x_0,x,x_0',x')=\int_0^\frac{t}{2} ds\int_{M^2}dzdz' P_{t-s}(x_0,z)P_s(z,x)P_{t-s}(x_0',z')P_s(z',x')\bm{G}_{\alpha,\rho}(z,z').$$

Since $0<s<t/2$, we have for all $t\geq \varepsilon$
    \begin{align}
        \frac{G_{t-s}(z,x)}{G_t(x_0,x)}=&\frac{(2\pi (t-s))^{-d/2}e^{-\frac{\bm{d}(z,x)^2}{2(t-s)}}+{C_H(t-s)}\wedge 1}{(2\pi t)^{-d/2}e^{-\frac{\bm{d}(x_0,x)^2}{2t}}+C_H(t\wedge 1)}\nonumber\\
        &\leq \frac{(\pi t)^{-d/2}+C_H}{(2\pi t)^{-d/2}e^{-\frac{R^2_M}{2\varepsilon}}+C_H}
        \leq C_2.\label{C_epsilon}
    \end{align}
    In the above $R_M$ is the diameter of $M$ and $C_2$ is a positive constant depending on $\varepsilon$ and $M$.  
    Now applying the heat kernel upper bound \eqref{def: G_t} together with \eqref{C_epsilon} and Proposition \ref{Prop: G_alpha} gives us 
    \begin{align}\label{upper bound L_1 midstep 1}\tilde{\mathcal{L}}_1&\leq C_3C_2^2 G_t(x_0,x)G_t(x_0',x')\int_0^\frac{t}{2}ds\iint_{M^2}dzdz' \bm{d}(z,z')^{2\alpha-d}G_s(z,x)G_s(z',x')\nonumber\\
  &\leq C_3C_2^2 G_t(x_0,x)G_t(x_0',x')\int_0^\frac{t}{2}k_1(s)ds .\end{align}

This gives \eqref{est t/2}.  

\bigskip
In order to show \eqref{est: int k1},  we write 
    \begin{align*}&G_s(z,x)G_s(z,x')\\
    \leq&(2\pi s)^{-d}e^{-\frac{\bm{d}(z,x)^2}{2s}}e^{-\frac{\bm{d}(z',x')^2}{2s}}+C\left[(2\pi s)^{-d/2}e^{-\frac{\bm{d}(z,x)^2}{2s}}+(2\pi s)^{-d/2}e^{-\frac{\bm{d}(z',x')^2}{2s}}\right]+C^2,\end{align*}
    which gives us $$\iint_{M^2}dzdz' \bm{d}(z,z')^{2\alpha-d}G_s(z,x)G_s(z',x')\leq C(I_1+I_2+I_3),$$ where \begin{align*}
        I_1:=&\sup_{x,x'\in M}\iint_{M^2} dzdz' (2\pi s)^{-d}e^{-\frac{\bm{d}(z,x)^2}{2s}}e^{-\frac{\bm{d}(z',x')^2}{2s}}\bm{d}(z,z')^{2\alpha-d},\\
        I_2:=&\sup_{x\in M}\iint_{M^2} dzdz' (2\pi s)^{-d/2}e^{-\frac{\bm{d}(z,x)^2}{2s}}\bm{d}(z,z')^{2\alpha-d},\\
        I_3:=&\iint_{M^2} dzdz' \bm{d}(z,z')^{2\alpha-d}.
    \end{align*}.

    Clearly 
    $$I_3= C^2\iint_{M^2}dzdz' \bm{d}(z,z')^{2\alpha-d}\leq C^2m_0c_{\alpha,M} = C_{1,\alpha,M}, \quad\mathrm{for\ all}\ s>0.$$
Here we used \eqref{L1 bound for covar} for the integral over $M^2$.
    
    An upper bound for $I_2$ is straightforward as well:
    \begin{align*}
       I_2\leq & 2C\sup_{x\in M} \iint_{M^2}dzdz' \bm{d}(z,z')^{2\alpha-d} (2\pi s)^{-d/2}e^{-\frac{\bm{d}(z,x)^2}{2s}}\\
        \leq&2 Cc_{\alpha,M}   \sup_{x\in M}\int_M dz (2\pi s)^{-d/2}e^{-\frac{\bm{d}(z,x)^2}{2s}}\\
      %  =&2\sup_{x\in M} C_\alpha \int_0^{t}ds (s^{d}\wedge 1) \left(\int_{B(x,\delta)}+\int_{B(x,\delta)^c}\right)dz (2\pi s)^{-d/2}e^{-\frac{\bm{d}(z,x)^2}{2s}}\\
       % \lesssim& \int_0^{t}ds (s^{d}\wedge 1)\left( \int_{B_{\mathbb{R}^d}(0,\delta)} (2\pi s)^{-d/2}e^{-\frac{\abs{z}^2}{2s}}dz +(2\pi s)^{-d/2}e^{-\frac{\delta^2}{2s}}\right)\\
        \leq&  2Cc_{\alpha,M} c_M= C_{2,\alpha,M},\quad\mathrm{for \ all}\ s>0.
    \end{align*}
    %{\color{blue}We need a remark once and for all in order to switch the integral from a small ball on the manifold to that on $\mathbb{R}^d$.}
    In the above, we used \eqref{L1 bound for covar} for the second inequality and \eqref{int Gaussian M} for the third. %by taking normal coordinates for the integral over $B(x,\delta)$ and uniformly bounding the Jacobian with curvature bounds. 
   % Thanks to the fact that $M$ is compact, this esitate is uniform in $x\in M$. %Throughout the rest of the paper, we will estimate integrals over small balls with integrals over $\R^d$ in the same way without further mentioning.
    %It remains to show that $k_1(s)$ is integrable near $s=0$ for an effective bound.
   
    \smallskip
    Finally, we estimate $I_1$. 
    Let $U_1=B(x,\delta)$ and $U_2=B(x',\delta)$,
    we further decompose the integral over $M^2$ into 4 parts: 
     \begin{align}&\iint_{M^2}(2\pi s)^{-d}e^{-\frac{\bm{d}(z,x)^2}{2s}}e^{-\frac{\bm{d}(z',x')^2}{2s}}\bm{d}(z,z')^{2\alpha-d}dzdz'\nonumber\\
    =&\iint_{U_1\times U_2}+\iint_{U_1^c\times U_2}+\iint_{U_1\times U_2^c}+\iint_{U_1^c\times U_2^c}=J_1(s)+J_2(s)+J_3(s)+J_4(s).\label{decomp I_1}\end{align}

Since $\bm{d}(x,z),\bm{d}(x,z')\geq \delta$ when $z$ and $z'$ are outside the corresponding balls,  $J_4(s)$ can be trivially bounded from above as follows,
    \begin{align}\label{est: J_4}J_4(s)\leq \iint_{U_1^c\times U_2^c} dzdz' \left((2\pi s)^{-\frac{d}{2}}e^{-\frac{\delta^2}{2s}}\right)^2 \bm{d}(z,z')^{2\alpha-d}\leq \left((2\pi s)^{-\frac{d}{2}}e^{-\frac{\delta^2}{2s}}\right)^2 m_0 c_{\alpha,M}.\end{align}
    Here, we have used \eqref{L1 bound for covar} for the last inequality.

    Utilizing $\bm{d}(z,x)\geq \delta$, together with \eqref{L1 bound for covar} and \eqref{int Gaussian M}, we also have %{\color{blue}More details for the last inequality.}
    \begin{align}\label{est: J2}J_2(s)\leq (2\pi s)^{-\frac{d}{2}}e^{-\frac{\delta^2}{2s}}\int_{U_2}(2\pi s)^{-\frac{d}{2}} e^{-\frac{\bm{d}(z',x')^2}{2s}} \norm{\bm{d}(\cdot,z')^{2\alpha-d}}_{L^1(M)}dz'\leq(2\pi s)^{-\frac{d}{2}} e^{-\frac{\delta^2}{2s}}c_M c_{\alpha,M}.\end{align}
  %  Here we have used  \eqref{L1 bound for covar} for the last inequality, and the integral over the ball $U_2=B(x',\delta)$ is handeled in the same way as before using normal coordinates around $x'$. 
    It is clear that $J_3(s)$ can be treated similarly.\\

    %For $I_1(s)$, we have chosen $\delta=i_M/8$, so for {\color{blue}$\bm{d}(x,x')\geq i_M/2$ (this is not right.  We do allow that $x=x'$)},\\
    %{\color{red}perhaps I worded it badly, but this is case 1, I handle the case of when they are closer right after.}\\
    The estimate for $J_1(s)$ takes more effort and is obtained differently according to $\bm{d}(x,x')\geq 5i_M/16$ or $\bm{d}(x,x')<5i_M/16$.
    
     When $\bm{d}(x,x')\geq 5i_M/16$, one has
     \begin{equation}
     \bm{d}(z,z')\geq \frac{i_M}{16}\quad \mathrm{for}\ \  z\in U_1,z'\in U_2, 
     \end{equation}
    which, together with \eqref{int Gaussian M}, implies 
    \begin{align}\label{d(x,x')large}
    J_1(s)\leq \iint_{U_1\times U_2}dzdz' (2\pi s)^{-d}e^{-\frac{\bm{d}(z,x)^2}{2s}}e^{-\frac{\bm{d}(z',x')^2}{2s}}\left(\frac{i_M}{16}\right)^{2\alpha-d}\leq c_M^2\left(\frac{i_M}{16}\right)^{2\alpha-d}.
    \end{align}
    
    On the other hand, when $\bm{d}(x,x')<5_M/16$ and $z\in U_1, z'\in U_2$, one has 
    \begin{align*}z'\in B(x',\delta)\implies \bm{d}(x,z')\leq \bm{d}(x,x')+\bm{d}(x',z)\leq \frac{5i_M}{16}+\frac{i_M}{8}=\frac{7i_M}{16}<\frac{i_M}{2}.
    \end{align*}
    That is,
\begin{align*}B(x',\delta)\subset B(x,i_M/2).\end{align*}
    Therefore 
    \begin{align}\label{est: J_1}J_1(s)\leq \iint_{B(x,i_M/2)\times B(x,i_M/2)}dzdz' (2\pi s)^{-d}e^{-\frac{\bm{d}(z,x)^2}{2s}}e^{-\frac{\bm{d}(z',x')^2}{2s}}\bm{d}(z,z')^{2\alpha-d}.\end{align}
    %The RHS integral can be bounded in coordinates using the result in [Chen Kim 2019] {\color{blue}strictly speaker, it is not a direct result of their paper. We may want to reproduce the proof here, even though it is pretty easy.}, which also tells us it is $O(s^{\frac{2\alpha-d}{2}})$.\\
    We will compute the right-hand side of \eqref{est: J_1} in local coordinates. We choose normal coordinates $z=(z_1,\dots,z_d)$ around $x=(0,\dots,0)$.
    For any $z, z'\in B(x,i_M/2)$, denote by 
    $$\abs{z-z'}=\left((z_1-z_1')^2+\cdots+(z_d-z_d')^2\right)^{1/2},$$ 
    so that  $\bm{d}(z,x)=|z-x|$. Note that $M$ being compact gives a uniform bound on the volume form in \eqref{est: J_1}. Moreover, non-positive curvature implies $d(z,x')\geq \abs{z-x'},\bm{d}(z,z')\geq \abs{z-z'}$. Indeed, for $y,y'\in B(x,i_M/2)$, we have $d(y,x)=\abs{y-x},d(y',x)=\abs{y'-x}$. Then $d(y,y')\geq \abs{y-y'}$ follows immediately from $d(y,y')^2\geq d(y,x)^2+d(y',x)^2-2d(y,x)d(y',x)\cos(\alpha_{yy'})=\abs{y-y'}^2$, where $\alpha_{yy'}$ is the angle made by the geodesics connecting $x,y$ and $x,y'$ (see \cite[Chapter 6]{PetersenRG}). With all the considerations above, when estimating the right-hand side of \eqref{est: J_1} in coordinates we can replace all Riemannian distances by $\abs{\cdot}$, and the integral in \eqref{est: J_1} is upper bounded (up to a multiple of a constant depending on $M$) by 
    $$\sup_{x,x'\in\mathbb{R}^d}\iint_{\R^{2d}}dzdz' (2\pi s)^{-d}e^{-\frac{\abs{z-x}^2}{2s}}e^{-\frac{\abs{z'-x'}^2}{2s}}\abs{z-z'}^{2\alpha-d}.$$ Standard Fourier analysis shows that the above is finite when $\alpha>(d-2)/2$ and that the supremum is achieved at $x=x'$. In particular, if we pick $x=x'=0$, a change of variables $y=z/\sqrt{s}, y'=z'/\sqrt{s}$ together with some elementary computation gives
    $$\sup_{x,x'\in\mathbb{R}^d}\iint_{\R^{2d}}dzdz' (2\pi s)^{-d}e^{-\frac{\abs{z-x}^2}{2s}}e^{-\frac{\abs{z'-x'}^2}{2s}}\abs{z-z'}^{2\alpha-d}\leq  Cs^{\frac{2\alpha-d}{2}}.$$
    Combining with \eqref{d(x,x')large}, we have shown 
    \begin{align}\label{est: J1}J_1(s)\leq C_4\left(1+s^{\frac{2\alpha-d}{2}}\right),\end{align}
    for some constant $C_4>0$ depending on $M$.%\iint_{B(x,i_M)\times B(x,i_M)}dzdz' (2\pi s)^{-d}e^{-\frac{\bm{d}(z,x)^2}{2s}}e^{-\frac{\bm{d}(z',x')^2}{2s}}\bm{d}(z,z')^{2\alpha-d}=O(s^{\frac{2\alpha-d}{2}}).$$

Now inserting estimates \eqref{est: J_4}, \eqref{est: J2} and \eqref{est: J1} into \eqref{decomp I_1}, we obtain

   %  To summarize, we have shown that for all $s\geq0$,
   %  \begin{align*}
   %      &J_1(s)\approx 1+O(s^{\frac{2\alpha-d}{2}}),\\
   %      &J_2(s),\ I_3(s)=O(s^{-d/2}e^{-\delta^2/(2s)}),\\
   %      &J_4(s)=O(s^{-d} e^{-\delta^2/s}), 
   %  \end{align*} 
   %  which implies
     $$I_1\leq C_M(1+s^{\frac{2\alpha-d}{2}}),\quad\mathrm{for\ all}\ s>0.$$
     which together with the estimates $I_2\leq C_{2,\alpha,M}$ and $I_3\leq C_{2,\alpha,M}$ for all $s>0$ completes the proof.
\end{proof}
%\begin{remark}
%    The integrability of $I_1(s)$ comes from $ I_1(s)=O(s^{\frac{2\alpha-d}{2}})$, thus $$\int_0^{t} k_1(s)ds=O(t^{\frac{2\alpha-d}{2}+1}\vee t).$$ This will be useful for obtaining the exponential upper bound for the $p$-th moment.
%\end{remark}

\subsection{Upper Bound for $\mathcal{L}_1$: $t<\varepsilon$}
%Since the $C_\varepsilon$ obtained at the beginning of the previous section goes to $+\infty$ as $\varepsilon\downarrow0$, it is necessary for us to use a different strategy for $t<\varepsilon$. 

 We now turn our attention to the estimate of $\mathcal{L}_1$ (more specifically \eqref{main quantity for interation G}) in small time. This is where the global geometry of $M$ starts to play a role and the curvature condition is used. The main result of this section is Theorem \ref{theorem:short time L1 bound}.    \\

In order to state the main result of this section, we need to introduce some notation. Recall the definition of $G_t(x,y)$ and $G_{t,x,y}(s,z)$ in \eqref{def: G_t(x,y)} and \eqref{def: G_{t,x,y}(s,z)} respectively, we have
\begin{align}\label{decomp of G}
    G_{t,x_0,x}(s,z)=&\frac{G_{t-s}(x_0,z) G_s(z,x)}{G_t(x_0,x)}\nonumber\\%&\frac{\bigg((2\pi(t-s))^{-\frac{d}{2}})e^{-\frac{\bm{d}(x_0,z)^2}{2(t-s)}}+C((t-s)^d\wedge1)\bigg)\bigg((2\pi s)^{-\frac{d}{2}}e^{-\frac{\bm{d}(z,x)^2}{2s}}+C(s^d\wedge1)\bigg)}{(2\pi t)^{-\frac{d}{2}}e^{-\frac{\bm{d}(x_0,x)^2}{2t}}+C(t^d\wedge1)}\\
    \leq& \left(2\pi\frac{s(t-s)}{t}\right)^{-\frac{d}{2}}\exp{\left(\frac{\bm{d}(x_0,x)^2}{2t}-\frac{\bm{d}(x_0,z)^2}{2s}-\frac{\bm{d}(z,x)^2}{2(t-s)}\right)}\nonumber\\
    \quad &+\frac{\text{terms with one or no Gaussians}}{C}\nonumber\\
    %\leq&\left(2\pi\frac{s(t-s)}{t}\right)^{-\frac{d}{2}}\exp{\left(\frac{\bm{d}(x_0,x)^2}{2t}-\frac{\bm{d}(x_0,z)^2}{2s}-\frac{\bm{d}(z,x)^2}{2(t-s)}\right)}\\
  %  &+\frac{1}{t^d\wedge 1}\left[(2\pi(t-s))^{-\frac{d}{2}}e^{-\frac{\bm{d}(x_0,z)^2}{2(t-s)}}(s^d\wedge 1)+(2\pi s)^{-\frac{d}{2}}e^{-\frac{\bm{d}(z,x)^2}{2s}}((t-s)^d\wedge 1)\right]+O(t^d\wedge 1)\\
    \leq&\left(2\pi\frac{s(t-s)}{t}\right)^{-\frac{d}{2}}\exp{\left(\frac{\bm{d}(x_0,x)^2}{2t}-\frac{\bm{d}(x_0,z)^2}{2s}-\frac{\bm{d}(z,x)^2}{2(t-s)}\right)}\nonumber\\
    &+\left((2\pi(t-s))^{-\frac{d}{2}}e^{-\frac{\bm{d}(x_0,z)^2}{2(t-s)}}+(2\pi s)^{-\frac{d}{2}}e^{-\frac{\bm{d}(z,x)^2}{2s}}\right)+C_H\nonumber\\
    :=\,&\Xi_{t,x_0,x}(s,z)+f_{t,x_0,x}(s,z)+C_H.
\end{align}
In the sequel, to lighten the notation,  whenever there is no confusion we will use $\Xi(*)$ and $\Xi(*')$ (respectively, $f(*)$ and $f(*')$) for $\Xi_{t,x,y}(s,z)$ (respectively, $f_{t,x,y}(s,z)$) depending on the space variables being $x,y,z$ or $x',y',z'$. With this notation, we have
\begin{align}\label{GG bound by Xi f}
    G_{t,x_0,x}(s,z)G_{t,x_0',x'}(s,z')\leq & \Xi(*)\Xi(*')+f(*)f(*')+C_H^2\nonumber\\
    &+\Xi(*)f(*')+\Xi(*')f(*)+C_H[f(*)+f(*')+\Xi(*)+\Xi(*')].
\end{align}
    Denote the right-hand side of \eqref{GG bound by Xi f} by $R_{t,x_0,x,x_0',x'}(s,z,z')$. \\

 %   We now state the main result of this section. 
\begin{theorem}\label{theorem:short time L1 bound}
    Assume $\frac{d}{2}>\alpha>\frac{d-2}{2}$.
    Define for each $s>0$,$$k_2(s):=\sup_{t\geq 2s}\sup_{x_0,x,x_0',x'\in M}\iint_{M^2 }dzdz'R_{t,x_0,x,x_0',x'}(s,z,z')\bm{d}(z,z')^{2\alpha-d}.$$
    If $M$ has non-positive sectional curvature, then
 $$k_2(s)\leq C_M(1+s^{\frac{2\alpha-d}{2}}),\quad\mathrm{for\ all}\ s>0,$$
for some positive constant depending on $M$. In addition, for all $t>0$
$$\mathcal{L}_1(t,x_0,x,x_0',x')\leq C_S G_t(x_0,x)G_t(x_0',x')\left(\int_0^t k_2(s)ds\right),$$
    where $C_S$ depends on $\alpha$ and $M$. %and $N_M$, the maximum number of geodesics between points in $M$ of length at most $2R_M$, where $R_M$ is the diameter of $M$. %In particular, $k_2(s)\sim_{s\downarrow 0}s^{\frac{2\alpha-d}{2}}$.
\end{theorem}

\begin{remark}
The upper bound of $\mathcal{L}_1$  claimed in Theorem \ref{theorem:short time L1 bound} is indeed valid for all $t>0$ (not only for small $t<\varepsilon$). From the analysis below, we expect it to be sharp for small time $t$. However, it happens to match the upper bound obtained in Theorem \ref{theorem:large time L1 bound} for large time as well.
\end{remark}

The proof of Theorem \ref{theorem:short time L1 bound} requires a good understand of how the measure of a Brownian bridge (more precisely, the measure given by $G_{t,x,y}(s,z)$) is concentrated. From the decomposition in \eqref{decomp of G}, it is clear that
the main difficulty stems from the term involving $\Xi(*)$:
%The trick allowing us to do only $\int_0^{\frac{t}{2}}$ still works.
%By our choice of $G_{t,x_0,x}(s,z)$, we need only show $\int_0^\frac{t}{2} k_2(s)ds$ is an effective bound.\\
 we need to carefully study the quantity in the exponential of $\Xi(*)$, which is the function $F$ given in \eqref{func: F}, 
    $$-F_{s,t;x,y}(z)=\frac{\bm{d}(x,y)^2}{2t}-\frac{\bm{d}(x,z)^2}{2s}-\frac{\bm{d}(z,y)^2}{2(t-s)}=
    \frac{1}{2\frac{s(t-s)}{t}}\left(\frac{s}{t}\cdot\frac{t-s}{t}\bm{d}(x,y)^2-\frac{t-s}{t}\bm{d}(x,z)^2-\frac{s}{t}\bm{d}(z,y)^2\right).$$
Let $a=s/t$, so we have $0<a<1.$ 
The term inside the parenthesis is thus \begin{align}\label{def: F}-F_a(z,x,y):=a(1-a)\bm{d}(x,y)^2-(1-a)\bm{d}(y,z)^2-a\bm{d}(z,x)^2.\end{align}
On $\R^d$, elementary computation shows that $F_a(z,x,y)=\bm{d}(z,w)^2$, where $w$ is the point on the line segment connecting $x$ and $y$ satisfying $d(x,w)=a\bm{d}(x,y)$. It implies that the Euclidean Brownian bridge is concentrated around $w$. One certainly should not expect such an identity to hold on a general manifold, which makes further analysis of $F$ necessary. %even some inequality of the form $F_a(z,x,y)\geq c\bm{d}(z,w)^2$ to hold globally since there are points for which minimizing geodesics on $M$ are not be unique. 
%In order to analyze $F_a(z,x,y)$ on a manifold, we will use some metric geometry, for which the sectional curvature of $M$ being non-positive is a crucial condition. 
The analysis of $F$ is tied to the global geometry of $M$. We perform this analysis in the next section for compact $M$ with non-positive sectional curvature.
%This is  one of the main novelties of the present paper.\\
%For succinctness, "geodesic" in the sequel will mean "distance minimizing geodesic" unless otherwise stated. 
\subsubsection{Preparation in geometry and topology}
%{\color{blue}The notation $[pq]$ should be introduced before the definition below.}

%This section is devoted to a careful analysis of $F_a(z,x,y)$. 
We start by recalling some notions and well-known facts in geometry and topology. Then we show that assuming non-positive sectional curvature, although there are infinitely many geodesics connecting $x$ and $y$, each geodesic is in a different homotopy class and there are only finitely many of them with length bounded by $L$ for any $L>0$. Moreover, $F_a(z,x,y)$  takes its minimum on those geodesics; thus the measure of $G_{t,x,y}(s,z)$ is concentrated around the minimums. More precise statement will be given in Lemmas \ref{Finite geodesic count} and \ref{behavior of $F$ in sausages} below. In what follows, we follow the convention that a geodesic $\gamma$ connecting $x$ and $y$ on $M$ is parametrized on $[0,1]$ with $\gamma(0)=x,\gamma(1)=y$.\\

%\begin{definition}\label{defn:Cat-K space}

\smallskip
%{\color{red}[Adopted from Bridson-Haefliger, II.1, definition 1.1]} 
Let $\Delta$ be a geodesic triangle connecting points $p,q,r$ in $M$. Suppose $\overline{\Delta}$ is a triangle with the same side lengths in $\R^2$ connecting points $\overline{p},\overline{q},\overline{r}$. Denote by $[pq]$ a geodesic connecting the points $p,q$. $\overline{x}\in [\overline{p}\,\overline{q}]$ is a {\it comparison point} of $x\in [pq]$ if $d(q,x)=\abs{\overline{q}-\overline{x}}$. Comparison points for other sides are defined similarly. We say $\Delta$ satisfies the {\it CAT(0) inequality} if for all $x,y\in \Delta$ and comparison points $\overline{x},\overline{y}\in \overline{\Delta}$, $$\bm{d}(x,y)\leq \abs{\overline{x}-\overline{y}}.$$
$M$ is a {\it CAT(0) space} if all geodesic triangles satisfy the CAT(0) inequality.

%{\color{blue}give a reference of the definition} {\color{red}This is adopted from Bridson-Haefliger (they give a more general definition).}
%It is well known that $\sec_M\leq 0\iff M$ is {\color{blue}locally CAT(0). (needs to explain what it means by ``localy CAT(0)").}
%\end{definition}
\begin{figure}[h]
    \centering
    \includegraphics[width=12cm, height=4cm]{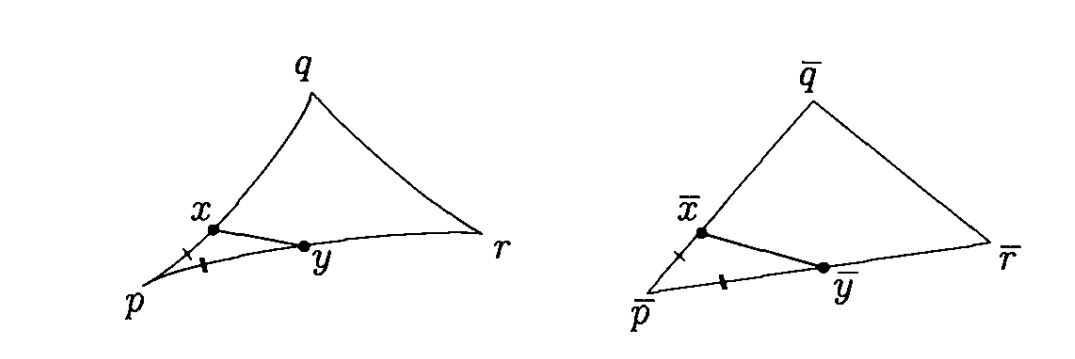}
    \caption{CAT(0) inequality, from \cite[Chapter II.1, figure 1.1]{BH}}
    \label{CAT_K Inequality}
\end{figure}

\begin{remark}\label{CAT(0) implies desired inequality}
Again we recall that on $\mathbb{R}^d$ one has $F_a(z,x,y)=|z-w|^2$, where $w$ is the point on the line segment connecting $x$ and $y$ satisfying $d(x,w)=a\bm{d}(x,y)$. Now let $\Delta$ be a geodesic triangle with vertices $x,y,z$ in a geodesic metric space $X$. If the CAT(0) inequality is satisfied by $\Delta$, applying it to $z$ and $w$, where $w\in[xy]$ satisfies $d(x,w)=a\bm{d}(x,y)$, we have $F_a(z,x,y)\geq \bm{d}(z,w)^2$.
\end{remark}

%The Cartan-Hadamard theorem(see [citation: Bridson-Haefliger II.4.1] for the version we are working with), tells us that the universal cover of $M$ is CAT(0) if $\sec_M\leq0$. 
%{\color{blue}(You need to emphasize that this is true under the condition $sec_M\leq0$. Also, if we only use the fact that the universal cover of $M$ is gloabally CAT(0), there is no need to mention the concept of ``locally CAT(0)".}. 
The above remark is the key observation that allows us to have a handle on $F_a(z,x,y)$. To see this, we need some facts about path homotopies taken from \cite[Appendix B]{JostRG}.\\

\smallskip
%\begin{definition}\label{homotopy of curves}
On a manifold $M$, paths $c_1,c_2:[0,1]\to M$ sharing the same endpoints are {\it homotopic} if there exists $H:[0,1]^2\to M$ continuous such that $H(t,0)=c_1(t)$ and $H(t,1)=c_2(t)$, $H(s,0)=c_1(0)=c_2(0)$ and $H(s,1)=c_1(1)=c_2(1)$. Denote by $c_1\simeq c_2$ if $c_1$ and $c_2$ are homotopic. It is clear that $\simeq$ gives an equivalence relation, and a homotopy class of curves consists of all curves in the homotopy equivalence class. Equivalence classes of homotopic paths with the same endpoints in $M$ form a group which does not depend on the choice of end points and is isomorphic to the group of homotopy-equivalent loops, which is %{\color{red}I believe fundamental group is formed by homotopy classes of loops} {\color{blue}the same group emerges regardless of endpoints, choosing the endpoints to be the same point gets you the group which is formed by loops, but they are the same. This is commonly known}, 
    called the {\it fundamental group} denoted by $\pi_1(M)$.
    It is well known that the fundamental group of a manifold is countable \cite[Theorem 7.21]{LeeTopMfd}. %{\color{red}(Lee, topological manifolds, theorem 7.21).}
    A space where the fundamental group is trivial is {\it simply connected}.\\
%\end{definition}

\smallskip

For two manifolds $M,\overline{M},$ a map $\pi:\overline{M}\to M$ is a {\it covering map} if for any $p\in M$, there exists a neighborhood $U_p$ of $p$ such that any connected component of $\pi^{-1}(U_p)$ is mapped homeomorphically onto $U_p$.  We say $\overline{M}$ is the {\it universal cover} of $M$ if $\overline{M}$ is simply connected. For any manifold, a universal cover is unique up to homeomorphism. Any path homotopy $H:[0,1]^2\to M$ lifts to a corresponding path homotopy $\overline{H}:[0,1]^2\to \overline{M}$ \cite[Proposition 1.30]{Hatcher}. A metric tensor on $M$ induces a metric tensor on $\overline{M}$, where the $\pi-$preimages of geodesics in $M$ are geodesics in $\overline{M}$ and $\pi$ becomes a local isometry \cite[Chapter I.3]{BH}. In particular, given a geodesic triangle $\triangle xyz$ on $M$ where the concatenation of two sides is homotopic to the third, there always exists a geodesic triangle $\triangle \overline{x}\,\overline{y}\,\overline{z}$ in $\overline{M}$ which is the pre-image of $\triangle xyz$ and the corresponding side lengths are the same. %{\color{red}(one needs to fix a lift of at least one vertex of the triangle in order to have the uniqueness of the list)}
%\end{definition}

The following Cartan-Hadamard Theorem is standard in differential geometry.

\begin{theorem}(Cartan-Hadamard)\label{Cartan-Hadamard}
    If a manifold $M$ of dimension $d$ admits a metric tensor satisfying $\sec_M\leq 0$, the following holds for its universal cover $\overline{M}$:
    \begin{enumerate}
        \item $\overline{M}$ is diffeomorphic to $\R^d$ via $\exp_{\overline{p}}$, the exponential map based at any point $\overline{p}\in \overline{M}$.
        \item $\overline{M}$ equipped with the induced metric tensor from $M$ is a CAT(0) space.
        \item For any $p\in M$, $\exp_p:\overline{M}\cong T_pM\to M$ is a covering map.
    \end{enumerate}
\end{theorem}
\begin{proof}
  %{\color{red}  Statement 1. is  is corollary 6.9.1 of Jost, statement 2. is theorem II.4.1 in Bridson-Haefliger, and statement 3. is theorem 12.8 of Lee, Riemannian manifolds.}
  See \cite[Corollary 6.9.1]{JostRG} for statement 1, \cite[Theorem II.4.1]{BH} for statement 2, and \cite[Theorem 12.8]{leeRmMfd} for statement 3.
\end{proof}
\begin{remark}\label{Universal Cover Distance}
    Let $\bm{d}$ be the distance function on $M$ and $\bm{\overline{d}}$ be the distance function associated with the induced metric on $\overline{M}$. Because $\exp_p$ is a local isometry for any $p\in M$, we must have $\bm{\overline{d}}(\overline{x},\overline{y})\geq \bm{d}(x,y)$ for any $x,y\in M$ and any two lifts $\overline{x}$ of $x$ and $\overline{y}$ of $y$. 
\end{remark}

\begin{lemma}\label{a geodesic exists in all homotopy classes}
    Suppose $M$ is a Riemannian manifold with non-positive sectional curvature.  Let $x,y\in M$. In every homotopy class of curves connecting $x$ and $y$, a unique geodesic exists and minimizes length over curves with endpoints $x,y$ in that homotopy class.
\end{lemma}
\begin{proof} See \cite[Theorem 6.9.1]{JostRG}.
\end{proof}

\begin{lemma}\label{Finite geodesic count}
Fix $L>0$. For $x,y\in M$,  denote by $N_L(x,y)$ the number of geodesics connecting $x$ and $y$ with length bounded by $L$. Assume $M$ has non-positive sectional curvature, then $0<N_L(x,y)<+\infty$. In addition, when $M$ is compact, $N_L(x,y)$ is uniformly bounded in $x,y\in M$.
\end{lemma}
%In the sequal, we will set $N(x,y)=N_{2R}(x,y)$ and $N_M=\sup_{x,y\in M}N_{2R_M}(x,y)$ the maximum number of geodesics of length no longer than $2R_M$ for any two points in $M$.

\begin{proof}
    Take a lift $\overline{x}$ of $x$ and consider $B(\overline{x},L)$ in $\overline{M}$. % (the choice of $\overline{x}$ does not matter because of covering space theory)
    Since $M$ has non-positive sectional curvature Cartan-Hadamard Theorem implies that each geodesic connecting $x$ and $y$ with length shorter than $L$ corresponds to a unique lift of $y$ in $B(\bar{x},L)$. Thus the first statement in the theorem is equivalent to bounding the number of lifts of $y$ inside $B(\overline{x},L)$. By the definition of the injectivity radius $i_M$, $2i_M$ is the shortest length of any geodesic loop. Thus for any two lifts $\overline{y}$, $\overline{y}'$ of $y$ we must have $\bm{\overline{d}}(\overline{y},\overline{y}')\geq 2i_M$. This implies that for any chosen lift $\overline{y}$ of $y$, $B(\overline{y},i_M)$ has no other lifts of $y$ in it. Thus lifts of $y$ in $B(\bar{x},L)$ are isolated, hence could only be finite.
   \vspace{2mm}
    
    The uniform bound (in $x$ and $y$) will be proved by contradiction and uses compactness of $M$. Suppose $\sup_{x,y\in M}N_L(x,y)=+\infty$, then there is a sequence $(x_n,y_n)\subset M\times M$ such that $N_L(x_n,y_n)\uparrow +\infty$ as $n$ tends to infinity.
 Since $M$ is compact, this sequence has at least one limit point which we denote by $(x,y)$. Without loss of generality, we assume $(x_n,y_n)\to (x,y)$. Pick and fix a lift $\overline{x}$ of $x$. In what follows, we will construct infinitely many lifts of $y$ in a closed ball centered at $\overline{x}$, which  contradicts the fact that all lifts of $y$ should be $2i_M$ apart.

 \vspace{2mm}
 
 First recall that $R_M$ is the diameter of $M$. All lifts of $y_n$ are inside a ball of radius $L+R_M$ centered at $\overline{x}$. On the other hand, since $(x_n,y_n)\to (x,y)$, we have $d(y_n,y)<i_M$ for $n>N$, where $N$ depends on $i_M$. We now show there are infinite many lifts of $y$ inside $\overline{B}(\overline{x},L+R_M+i_M)$. Indeed, since the covering map is locally isometric to $M$, for any fixed $n>N$ each lift of $y_n$ must correspond to a unique lift of $y$ at most $i_M$ away from its corresponding lift of $y_n$.  Moreover, since $n>N$, all these lifts of $y$ lie inside the ball $\overline{B}(\overline{x},L+R_M+i_M)$. By assumption, there are at least $N_L(x_n,y_n)$ number lifts of $y$ for each $n$, and $N_L(x_n,y_n)\uparrow\infty$. The proof is thus completed. 
\end{proof}

\begin{remark}
    Lemma \ref{Finite geodesic count} is false if we do not restrict the lengths of geodesics. For example, consider $(0,0),(\frac{1}{2},\frac{1}{2})\in\mathbb{T}^2=\R^2/\Z^2$ identified with $[0,1)\times [0,1)$. Then any line $y=\frac{1}{2^n}x,n\in \N$ in $\R^2$ produces a geodesic connecting $(0,0)$ to $(\frac{1}{2},\frac{1}{2})$ when projected down to $\mathbb{T}^2$. Obviously the lengths of these geodesic segments go to infinity as $n\uparrow+\infty$.
\end{remark}

\begin{definition}\label{sausages and restricted balls}
    For any $x,y\in M$, let $\Gamma_{xy}=\set{\gamma_i}_{i=1}^{N(x,y)}$  be the collection of geodesics up to length $2R_M$ connecting them, and denote by $\overline{xy}$ a (not necessarily unique) minimizing geodesic connecting them. For each $\gamma_i\in\Gamma_{xy}$,  the {Sausage $S^i_{xy}$ around $\gamma_i$} is defined by $$S^i_{xy}:=\set{z\in M: \mathrm{there\ exists\ }\overline{xz}\ \mathrm{and}\ \overline{zy}\text{ such that }\overline{xz}\sqcup\overline{zy}\simeq \gamma_i}.$$
    For $a\in[0,1]$ and $\delta={i_M}/{8}$, the {restricted ball around $\gamma_i(a)$} is defined by $B^i_{xy}(a,\delta):=B(\gamma_i(a),\delta)\cap S^i_{xy}$, and the {set outside the restricted ball in the sausage} will be denoted $C^i_{xy}(a,\delta):=S^i_{xy}\setminus B^i_{xy}(a,\delta)$.
\end{definition}

\begin{remark}\label{sausages cover $M$}
    For any fixed $x,y\in M$, since the lengths of minimizing geodesics $\overline{xz}$ and $\overline{zy}$ are bounded by $R_M$, $\overline{xz}\sqcup\overline{zy}$ has length no greater than $2R_M$.  Lemma \ref{a geodesic exists in all homotopy classes} then implies that any  $z\in M$ must be in $S^i_{xy}$ for some $i=1,\dots,N_M$. Thus $\set{S^i_{xy}}_{i=1}^{N(x,y)}$ covers $M$. %and we can bound integrals over $M$ by summing integrals over $S^i_{xy}$.
\end{remark}

\begin{remark}\label{topology of sausages}Thanks again to Lemma \ref{a geodesic exists in all homotopy classes}, each homotopy class contains a unique geodesic that minimizes distance among all curves in that homotopy class.
    Therefore a point not in the cut-locus of either $x$ or $y$ can only be in one sausage. %{\color{red}Thus all sausages have disjoint interiors, with the union of their boundaries being the union of the cut-loci of $x$ and $y$, so they are closed. (I believe all the statements here. But we need to have a proof of them even though we do not need to write the proof down in the paper.)} 
    This implies all sausages are measurable sets. Indeed, subtracting the cut-locus of $x$ and $y$ which has measure $0$, each sausage is an open set.  
\end{remark}

\begin{remark}\label{lifting relevant triangles in sausages}
In general, lifts of a triangle in $M$ may not be a triangle in $\overline{M}$. The construction of $S^i_{xy}$ ensures that for any $z\in S^i_{xy}$ the triangle formed by $\overline{xz}\sqcup\overline{zy}\sqcup\gamma_i$ can be lifted to a geodesic triangle $\triangle \overline{x}\,\overline{y}\,\overline{z}$ in $\overline{M}$ with the same side lengths. %in $\overline{M}$ and corresponding vertices $\overline{x},\overline{y},\overline{z}$. 
Since $\overline{M}$ is a $CAT(0)$ space by the Cartan-Hadamard theorem, $\triangle \overline{x}\,\overline{y}\,\overline{z}$ satisfies the $CAT(0)$ inequality. This fact is crucial for our analysis of $F_a(z,x,y)$.
\end{remark}

Now we can state the main result of this section. Recall the definition of $S^i_{xy}$ and $C^i_{xy}$ in Definition~\ref{sausages and restricted balls}.
\begin{lemma}\label{behavior of $F$ in sausages}
    Fix $x,y\in M$. For any $1\leq i\leq N(x,y)$, $z\in S^i_{xy}$, and $a\in(0,1)$ we have 
    \begin{align}\label{eq: hebavior of $F$ in sausages}F_a(z,x,y)\geq \max_{i:\,z\in S^i_{xy}}\bm{d}(z,\gamma_i(a))^2.
    \end{align}
    In particular, we have  $F_a(z,x,y)\geq \delta^2$ if $z\in C^i_{xy}(a,\delta)$ for some $i=1,\dots,N(x,y)$.
    %{\color{blue}Since $z$ could be in multiple sausages, the correct statement should be
    %$$F_z(z,x,y)\geq \min_i \bm{d}(z,\gamma_i(a))^2$$}
    %{\color{red}The above is true, but does not appear to be necessary. The proof below should be sufficient for the statement of the lemma as written, if it's correct. If we use this new version, the argument might become more complicated.}{\color{blue} I think now I get it; the correct statement should be $F_z(z,x,y)\geq \max_i \bm{d}(z,\gamma_i(a))^2$, correct? In addition, I think we need $\overline{xz}\sqcup\overline{zy}\simeq \gamma_i$ in order to be able to lift the triangle to a triangle in $T_xM$, correct? Otherwise, the lift may not be a triangle.} {\color{red}Yes, $F\geq \bm{d}(z,\gamma_i(a))^2$ if $z\in S^i_{xy}$, since $S^i_{xy}$ is defined by requiring $\overline{xz}\sqcup\overline{zy}\simeq \gamma_i$. The lift will be a triangle if this homotopy is satisfied.}
\end{lemma}
\begin{proof}
    The second statement follows trivially from the first, so it suffices to prove the first. 

    \smallskip
    Suppose $z\in S^i_{xy}$ for some $i=1,\dots,N(x,y)$. By the definition of $S^i_{xy}$, there exist minimizing geodesics $\overline{xz},\overline{zy}$ such that the $\overline{xz}\sqcup\overline{zy}\simeq \gamma_i$.   For any curve $\gamma:[0,1]\ra M$, denote by $L(\gamma)$  the length of $\gamma$. 
    Recall the definition of $F_a(z,x,y)$ in \eqref{def: F}. Since $L(\gamma_i)\geq \bm{d}(x,y)$, we have for every $a\in(0,1)$ 
\begin{align}\label{m step 1}
F_a(z,x,y)\geq (1-a)\bm{d}(x,z)^2+a\bm{d}(z,y)^2-a(1-a)L(\gamma_i)^2.
\end{align}
    Thanks to Remark \ref{lifting relevant triangles in sausages},  we can lift the geodesic triangle $\triangle xyz$ onto a geodesic triangle $\triangle\overline{x}\,\overline{y}\,\overline{z}$ of the same side lengths in $T_xM$, which is a $CAT(0)$ space. For each $a\in(0,1)$, let $\overline{w}_i(a)$ be the lift of $\gamma_i(a)$ to this triangle.
    By definition, $\bm{\overline{d}}(\overline{x},\overline{y})=L(\gamma_i)$. As noted in Remark \ref{lifting relevant triangles in sausages}, $\triangle\overline{x}\,\overline{y}\,\overline{z}$ satisfies the CAT(0) inequality. If we define $\overline{F}_a$ using the universal cover distance $\bm{\overline{d}}$ the same way as $F_a$:
\begin{align}\label{m step 2}
\overline{F}_a(\overline{z},\overline{x},\overline{y}):=(1-a)\bm{\overline{d}}(\overline{x},\overline{z})^2+a\bm{\overline{d}}(\overline{z},\overline{y})^2-a(1-a)\bm{\overline{d}}(\overline{x},\overline{y})^2,
\end{align}
we have the right hand side of \eqref{m step 1} equal to $\overline{F}_a(\overline{z},\overline{x},\overline{y})$. Hence
$$F_a(z,x,y)\geq \overline{F}_a(\overline{z},\overline{x},\overline{y}).$$
On the other hand, applying the CAT(0) inequality to $\triangle\overline{x}\,\overline{y}\,\overline{z}$ and Remark \ref{CAT(0) implies desired inequality} along with Remark \ref{Universal Cover Distance} gives us $$\overline{F}_a(\overline{z},\overline{x},\overline{y})\geq \bm{\overline{d}}(\overline{z},\overline{w}_i(a))^2\geq \bm{d}(z,\gamma_i(a))^2.$$
    The proof is now completed.
\end{proof}

\subsubsection{Proof of Theorem \ref{theorem:short time L1 bound}}

Recall
\begin{align*}\mathcal{L}_1(t,x_0,x,x_0',x')&=\int_0^t ds\int_{M^2}dzdz' P_{t-s}(x_0,z)P_s(z,x)P_{t-s}(x_0',z')P_s(z',x')\bm{G}_{\alpha,\rho}(z,z')\\
&\leq G_t(x,x_0)G_t(x',x_0')\int_0^t ds\int_{M^2}dzdz'  G_{t,x_0,x}(s,z) G_{t,x_0',x'}(s,z)\bm{G}_{\alpha,\rho}(z,z').
\end{align*}
As before, we can decompose the time integral into two parts $\int_0^t=\int_0^{t/2}+\int_{t/2}^t.$ We claim
\begin{align}\label{short time key bound}
\int_{M^2}dzdz'  G_{t,x_0,x}(s,z) G_{t,x_0',x'}(s,z)\bm{G}_{\alpha,\rho}(z,z')
\leq  k_2(s)\leq C_M(1+s^{\frac{2\alpha-d}{2}}),\quad \mathrm{for\ all}\ s\in (0,t/2).
\end{align}
Then by the symmetry between $s$ and $t-s$, and that between $x, x_0$ and $x', x_0'$, we can conclude 
$$\int_0^t ds\int_{M^2}dzdz'  G_{t,x_0,x}(s,z) G_{t,x_0',x'}(s,z)\bm{G}_{\alpha,\rho}(z,z')
\leq 2C_M\int_0^{\frac{t}{2}}k_2(s)ds\leq 2C_M \int_0^tk_2(s)ds,$$
which, together with \eqref{short time key bound},  gives the desired bound in Theorem \ref{theorem:short time L1 bound}.

\bigskip
In what follows we establish \eqref{short time key bound}.   Recall the decomposition in \eqref{GG bound by Xi f}, the rest of the proof is divided into two steps.

\bigskip

{\noindent\it Step 1:  Terms not involving $\Xi(*)$ or $\Xi(*')$.}

  The statement holds trivially for $C_H$. The terms of the form {$C_Hf$} can be treated similarly to $I_2$ in the proof of Theorem \ref{theorem:large time L1 bound}. %note that it behaves like the $(s^d\wedge 1)[(2\pi s)^{-d/2}e^{-\frac{\bm{d}(z,x)^2}{2s}}+(2\pi s)^{-d/2}e^{-\frac{\bm{d}(z',x')^2}{2s}}]$ of $G_s(z,x)G_s(z,x')$, thus integrates to $O(t^{d+1}\wedge t)$.\\
    This leaves us with the term $f(*)f(*')$, which we compute below:
    \begin{align}\label{est: ff small t}
        f(*)f(*')=&(2\pi (t-s))^{-d}e^{-\frac{\bm{d}(x_0,z)^2}{2(t-s)}}e^{-\frac{\bm{d}(x_0',z')^2}{2(t-s)}}+(2\pi s)^{-d}e^{-\frac{\bm{d}(z,x)^2}{2s}}e^{-\frac{\bm{d}(z',x')^2}{2s}}\nonumber\\
        &+(2\pi)^{-d}s^{-\frac{d}{2}}(t-s)^{-\frac{d}{2}}\left(e^{-\frac{\bm{d}(x_0,z)^2}{2(t-s)}}e^{-\frac{\bm{d}(z',x')^2}{2s}}+e^{-\frac{\bm{d}(x_0',z')^2}{2(t-s)}}e^{-\frac{\bm{d}(z,x)^2}{2s}}\right).
    \end{align}
    The second term in the above is the same as $I_1(s)$ in the proof of Theorem \ref{theorem:large time L1 bound}, and thus upper bounded by $C_M(s^{\frac{2\alpha-d}{2}}+1)$.  In addition, since we assumed $s<t/2$ (or, equivalently, $t>2s$) together with the fact that $2\alpha-d<0$, so does the first when taking the supremum over $t>2s$.\\
    %{\color{blue}For the third term, apply the decomposition $\int_0^t=\int_0^{t/2}+\int_{t/2}^t$ and then use $0\leq s\leq t/2\leq t-s\leq t$ for $s\in [0,t/2)$ and $0\leq t-s\leq t/2\leq s\leq t$ for $s\in [t/2,t]$ to bound both terms with the first two (we can do so because of the symmetry of roles of $x_0,x,x_0',x'$). (Expand this part.  More detail is needed. As I can see, the space integrals are treated the same way as other terms, it is the time integral, that one needs to have some special treatment, correct?)}
    Finally, we estimate the third term in \eqref{est: ff small t}. It suffices to show%Similar to before, using the decomposition $\int_0^t=\int_0^{t/2}+\int_{t/2}^t$ for the time integral and symmetry between $s$ and $t-s$, we only need to focus on 
     \begin{align}\label{est: mix kernel covariance}\sup_{t\geq 2s}\sup_{x_0,x'\in M}\iint_{M^2}dzdz'(2\pi)^{-d}s^{-\frac{d}{2}}(t-s)^{-\frac{d}{2}}e^{-\frac{\bm{d}(x_0,z)^2}{2(t-s)}}e^{-\frac{\bm{d}(z',x')^2}{2s}}\bm{d}(z,z')^{2\alpha-d}\leq C_M(s^{\frac{2\alpha-d}{2}}+1).\end{align}
    For this purpose, observe that for $s\in [0,t/2)$, we have $\frac{t}{2}<t-s$, which implies $(t-s)^{-\frac{d}{2}}e^{-\frac{\bm{d}(x_0,z)}{2(t-s)}}\leq (t/2)^{-\frac{d}{2}}e^{-\frac{\bm{d}(x_0,z)}{2t}}$. Hence
    \begin{align}\label{weird cross term}
        &\iint_{M^2}dzdz'(2\pi)^{-d}s^{-\frac{d}{2}}(t-s)^{-\frac{d}{2}}e^{-\frac{\bm{d}(x_0,z)^2}{2(t-s)}}e^{-\frac{\bm{d}(z',x')^2}{2s}}\bm{d}(z,z')^{2\alpha-d}\nonumber\\
       & \leq \int_M dz(t/2)^{-\frac{d}{2}}e^{-\frac{\bm{d}(x_0,z)}{2t}} (2\pi)^{-d} \int_M dz' s^{-\frac{d}{2}}e^{-\frac{\bm{d}(z',x')^2}{2s}}\bm{d}(z,z')^{2\alpha-d}\nonumber\\
      &  \leq C_{M}%\left(1+t^{-\frac{d}{2}}e^{-\frac{\delta}{2t}}\right) 
      \sup_{z\in M}\left\{ \int_M dz' s^{-\frac{d}{2}}e^{-\frac{\bm{d}(z',x')^2}{2s}}\bm{d}(z,z')^{2\alpha-d}\right\}.
    \end{align}
    Here we have used Remark \ref{rk: turn to eclidean integral} for the second inequality.
    For the spatial integral in \eqref{weird cross term}, we apply the decomposition for $\delta=i_M/8$, 
    \begin{align}\label{est: decompose M to two parts}\int_M=\int_{B(x',\delta)\cup B(z,\delta)}+\int_{M\setminus[B(x',\delta)\cup B(z,\delta)]}.\end{align}
    For the second integral above, since $z'\notin B(x',\delta)\cup B(z,\delta)$, the integrand is bounded by $\delta^{2\alpha-d}s^{-\frac{d}{2}}e^{-\frac{\delta^2}{2s}}$, and so does the integral thank to the fact that $M$ is compact.\\
    For the first integral, %$z'\in B(x',\delta)\cup B(z,\delta)$, 
    we divide by cases according to $\bm{d}(x',z)\geq \frac{5i_M}{16}$ and $\bm{d}(x',z)< \frac{5i_M}{16}$.\\
    {\it Case 1:} $\bm{d}(x',z)\geq \frac{5i_M}{16}$. \quad In this case,  $B(x',\delta)\cap B(z,\delta)=\emptyset$. Hence $$z'\in B(x',\delta)\implies \bm{d}(z,z')^{2\alpha-d}<(i_M/16)^{2\alpha-d},$$ while $$z'\in B(z,\delta)\implies e^{-\frac{\bm{d}(x',z')^2}{2s}}\leq e^{-\frac{(i_M/16)^2}{2s}}.$$
By Remark \ref{rk: turn to eclidean integral}, we  have
\begin{align}\label{est: easy term case 1}
  &\int_{B(x',\delta)\cup B(z,\delta)}dz' s^{-\frac{d}{2}}e^{-\frac{\bm{d}(z',x')^2}{2s}}\bm{d}(z,z')^{2\alpha-d}\nonumber\\
 \leq &\int_{B(z,\delta)}dz' s^{-\frac{d}{2}}e^{-\frac{\bm{d}(z',x')^2}{2s}}\bm{d}(z,z')^{2\alpha-d}+\int_{B(x',\delta)}dz' s^{-\frac{d}{2}}e^{-\frac{\bm{d}(z',x')^2}{2s}}\bm{d}(z,z')^{2\alpha-d}\nonumber\\
\leq& \int_{B(z,\delta)}dz's^{-\frac{d}{2}}e^{-\frac{(i_M/16)^2}{2s}}\bm{d}(z,z')^{2\alpha-d}+\int_{B(x',\delta)}dz' s^{-\frac{d}{2}}e^{-\frac{\bm{d}(z',x')^2}{2s}}(i_M/16)^{2\alpha-d}\nonumber\\
\leq & C_\alpha s^{-\frac{d}{2}}e^{-\frac{(i_M/16)^2}{2s}}+ C_M(i_M/16)^{2\alpha-d}.
\end{align}

    {\it Case 2:} $\bm{d}(x',z)< \frac{5i_M}{16}$. \quad We have for $z'\in B(x',\delta)\cup B(z,\delta)$, $$\bm{d}(x',z')<\frac{7i_M}{16}< \frac{i_M}{2},$$
    which implies
    $B(x',\delta)\cup B(z,\delta)\subset B(x',\frac{i_M}{2}).$ We then apply $$\int_{B(x',\delta)\cup B(z,\delta)}\leq \int_{B(x',\frac{i_M}{2})}$$
    and take normal coordinates at $x'=0$, changing all distance functions to Euclidean distances following the same considerations as used in treating $J_1(s)$  in the proof of Theorem \ref{theorem:large time L1 bound}. %Remembering that $$\sup_{z\in B_{\R^d}(0,\frac{i_M}{2})}\norm{\frac{1}{\abs{z-\cdot}^{d-2\alpha}}}_{L^1(B_{\R^d}(0,\frac{i_M}{2}))}\leq \norm{\frac{1}{\abs{\cdot}^{d-2\alpha}}}_{L^1(B_{\R^d}(0,i_M))}<\infty$$
    We then obtain
    \begin{align}\label{est: easy term case 2}
    &\int_{B_{\R^d}(0,\frac{i_M}{2})}dz' s^{-\frac{d}{2}}e^{-\frac{\abs{z'}^2}{2s}}\abs{z-z'}^{2\alpha-d}\nonumber\\
   &\leq \int_{\R^d}dz' s^{-\frac{d}{2}}e^{-\frac{\abs{z'}^2}{2s}}\abs{z-z'}^{2\alpha-d}\nonumber\\
   &\leq\int_{\R^d}dz's^{-\frac{d}{2}}e^{-\frac{\abs{z'}^2}{2s}} \abs{z'}^{2\alpha-d}\leq Cs^{\frac{2\alpha-d}{2}},
    \end{align}
    where the last equality is obtained by a change of variable $w=z'/\sqrt{s}$.\\
    Putting together the considerations from \eqref{est: decompose M to two parts} to \eqref{est: easy term case 2}, we conclude 
    \begin{align}\label{est: int Gaussian and convariance}\int_Mdz' s^{-\frac{d}{2}}e^{-\frac{\bm{d}(z',x')^2}{2s}}\bm{d}(z,z')^{2\alpha-d}\leq C_{M,\delta} s^{-\frac{d}{2}}[e^{-\frac{\delta^2}{2s}}+e^{-\frac{(i_M/16)^2}{2s}}]+1+s^{\frac{2\alpha-d}{2}}.
    \end{align}
    
    It is clear that the right-hand side of \eqref{est: int Gaussian and convariance} is independent of the choice of $x'$ and upper bounded by $C_M(s^{\frac{2\alpha-d}{2}}+1)$ for a proper choice of $C_M>0$. The proof of \eqref{est: mix kernel covariance} is thus completed. 
  %  The only term which blows up as $s\downarrow 0$ is $s^{\frac{2\alpha-d}{2}}$, which is integrable in this limit by our assumption $\alpha>\frac{d-2}{2}$.    
 %   This completes the proof.

\bigskip

{\noindent\it Step 2. Terms involving $\Xi(*)$ and $\Xi(*')$.}

First recall that $R_M$ is the radius of $M$.  For $x_0,x\in M$, set $n=N_{2R_M}(x_0,x)$ the number of geodesics connecting $x_0$ and $x$ with length no longer than $2R_M$, and denote by $\Gamma_{x_0x}=\set{\gamma_i}_{i=1}^{n}$ the collection of such geodesics. % and for each $a=\frac{s}{t}\in (0,1)$, $1\leq 1\leq n$, define $w_i(a)\in \gamma_i$ by $d(x_0,w_i(a))=a\bm{d}(x_0,x)$. 
    For each $1\leq i\leq n$,  $S^i_{x_0x}$, $B^i_{x_0x}(s/t,\delta)$, and $C^i_{x_0x}(s/t,\delta)$ are introduced in Definition \ref{sausages and restricted balls}. To lighten the notation, we will use $S^i, B^i(s)$ and $C^i(s)$ when there is no confusion.
For $x_0',x'\in M$,  $\Gamma_{x_0'x'}=\set{\eta_j}_{j=1}^m$,   $S^j_{x_0'x'}$, $B^j_{x_0'x'}(s/t,\delta)$, and $C^j_{x_0'x'}(s/t,\delta)$ (as well as ${S^{j}}', B^j(s)'$ and $C^j(s)'$) are defined analogously. 
 Note that  both $n$ and $m$ are uniformly bounded above thanks to Lemma \ref{Finite geodesic count}.   We finally emphasize that since we assume $s\in (0,t/2)$ in \eqref{short time key bound} one has $\frac{t-s}{t}\geq \frac{1}{2}$ which will be used repeatedly below.
  %  \tb{Notation:} For $x_0,x\in M$, let $\Gamma_{x_0x}=\set{\gamma_i}_{i=1}^{n}$ (so $N(x_0,x)=n$).% and for each $a=\frac{s}{t}\in (0,1)$, $1\leq 1\leq n$, define $w_i(a)\in \gamma_i$ by $d(x_0,w_i(a))=a\bm{d}(x_0,x)$. 
   % For each $1\leq i\leq n$ and $s\in(0,t/2),$ recall $S^i_{x_0x}$, $B^i_{x_0x}(s/t,\delta)$, and $C^i_{x_0x}(s/t,\delta)$ as in definition \ref{sausages and restricted balls}.\\
    \\
   % For notational simplicity, since $\delta$ is fixed and for fixed $t$ the map $s\mapsto s/t$ is a bijection, we associate the index $i$ with the points $x_0,x$ and $j$ associated with $x_0',x'$, we will write $S^i,B^i(s),C^i(s)$ and $S^j,B^j(s),C^j(s)$ leaving out the subscripts and $t$ in the sequel. {\color{blue}$S_i$ and $S'_j$ etc are better notation}\\
%    For the space integrals, when appropriate we will use the following bounds allowed by remark \ref{sausages cover $M$}: $$\int_M\leq\sum_{i=1}^n\int_{S^i},\sum_{j=1}^m \int_{S^j}.$$
   
    \bigskip
% Recall the decomposition in \eqref{GG bound by Xi f}, thanks to Lemma \ref{lemma:Easy terms}, we only need to bound the integrals involving $\Xi(*)$ and $\Xi(*')$. 
   
By symmetry of the roles between  $x_0,x$ and $x_0',x'$, we need only bound three types of integrals listed below:\begin{enumerate}[label=(\roman*)]
    \item  $\iint_{M^2}dzdz' \Xi(*)\bm{d}(z,z')^{2\alpha-d}$,
    \item $ \iint_{M^2}dzdz' \Xi(*)f(*')\bm{d}(z,z')^{2\alpha-d}$,
    \item $ \iint_{M^2}dzdz' \Xi(*)\Xi(*')\bm{d}(z,z')^{2\alpha-d}$.
    \end{enumerate}
   
    For integral (i), by \eqref{L1 bound for covar} we have
    \begin{align*}
        & \iint_{M^2}dzdz' \Xi(*)\bm{d}(z,z')^{2\alpha-d}\\
        =& \iint_{M^2}dzdz' \left(2\pi\frac{s(t-s)}{t}\right)^{-\frac{d}{2}}\exp\left\{\frac{-F_{s/t}(z,x_0,x)}{2\frac{s(t-s)}{t}}\right\}\bm{d}(z,z')^{2\alpha-d}\\
        \leq & c_{\alpha,M} \int_M dz \left(2\pi\frac{s(t-s)}{t}\right)^{-\frac{d}{2}}\exp\left\{\frac{-F_{s/t}(z,x_0,x)}{2\frac{s(t-s)}{t}}\right\}\\
        \leq & c_{\alpha,M} \sum_{i=1}^n\int_{S^i} dz \left(2\pi\frac{s(t-s)}{t}\right)^{-\frac{d}{2}}\exp\left\{\frac{-F_{s/t}(z,x_0,x)}{2\frac{s(t-s)}{t}}\right\}.
\end{align*}
Thanks to Lemma \ref{behavior of $F$ in sausages} and \eqref{int Gaussian M}, for each $1\leq i\leq n$, the space integral in the summation above is further bounded by
\begin{align*}
          \int_{S^i} dz\left(2\pi\frac{s(t-s)}{t}\right)^{-\frac{d}{2}}\exp\left\{\frac{-\bm{d}(z,\gamma_i(s/t))^2}{2\frac{s(t-s)}{t}}\right\} \leq   c_{M}.
      %  \leq&  \left(\int_{B^i(s)}+\int_{C^i(s)}\right)dz \left(2\pi\frac{s(t-s)}{t}\right)^{-\frac{d}{2}}\exp\left\{\frac{-d(z,\gamma_i(s/t))^2}{2\frac{s(t-s)}{t}}\right\}\\
      %  \leq&\int_{B(\gamma_i(s/t),\delta)}dz\left(2\pi\frac{s(t-s)}{t}\right)^{-\frac{d}{2}}\exp\left\{\frac{-d(z,\gamma_i(s/t))^2}{2\frac{s(t-s)}{t}}\right\}+m_0 \left(2\pi\frac{s(t-s)}{t}\right)^{-\frac{d}{2}}e^{\frac{-\delta^2}{2\frac{s(t-s)}{t}}}
         %n+(n+1)m_0\left(2\pi\frac{s(t-s)}{t}\right)^{-\frac{d}{2}}e^{\frac{-\delta^2}{2\frac{s(t-s)}{t}}}%=\frac{n}{2}t+O(e^{-\delta^2/t})%\leq \frac{N_M}{2}t+O(e^{-\delta^2/t}).
    \end{align*}
     Since the above bound is uniform in $x,x_0,x',x'_0$ and $t\geq s$, we conclude that for all $s>0$,
\begin{align}\label{small time first term}
\sup_{t\geq 2s}\sup_{x_0,x,x_0',x'\in M}\iint_{M^2}dzdz' \Xi(*)\bm{d}(z,z')^{2\alpha-d}\leq C_M.
\end{align}

\bigskip

 %   In the second to last line, lemma \ref{behavior of $F$ in sausages} was used for the $C^i(s)$ integral along with $Vol(C^i(s))\leq m_0$, $B^i(s)\subset B(\gamma_i(s/t),\delta)$.
  %  In the last line, we used $\delta<i_M$ and took normal coordinates at $\gamma_i(s/t)$ in each summand, bounded the Jacobian using curvature, and used the integrand is Gaussian pdf on $\R^d$.\\
    For integral (ii),  since $s\in (0,t/2)$ we have
    $$(2\pi(t-s))^{-\frac{d}{2}}e^{-\frac{\bm{d}(x_0,z)^2}{2(t-s)}}\leq (\pi t)^{-\frac{d}{2}}e^{-\frac{\bm{d}(x_0,z)^2}{2t}}.$$ 
   Recalling the defintion of $f(*)$ in \eqref{GG bound by Xi f},   integral (ii) is bounded above by
   \begin{align}\label{short time key term (2)} \iint_{M^2}dzdz' \Xi(*)\left[(2\pi s)^{-\frac{d}{2}}e^{-\frac{\bm{d}(z',x')^2}{2s}}+(\pi t)^{-\frac{d}{2}}e^{-\frac{\bm{d}(x_0',z')^2}{2t}}\right]\bm{d}(z,z')^{2\alpha-d}.\end{align}
    An estimate of the Gaussian term without $s$ is straightforward, and can be obtained as follows,
    \begin{align*}
        & \iint_{M^2}dzdz' \Xi(*)(\pi t)^{-\frac{d}{2}}e^{-\frac{\bm{d}(x_0',z')^2}{2t}}\bm{d}(z,z')^{2\alpha-d}\\
        =&\int_M dz' (\pi t)^{-\frac{d}{2}}e^{-\frac{\bm{d}(x_0',z')^2}{2t}}\int_M dz\Xi(*)\bm{d}(z,z')^{2\alpha-d}\\
        \leq& \int_M dz' (\pi t)^{-\frac{d}{2}}e^{-\frac{\bm{d}(x_0',z')^2}{2t}}\left(\sum_{i=1}^n \int_{S^i}\right)dz \Xi(*)\bm{d}(z,z')^{2\alpha-d}\\
       \leq & C_M\sup_{z'\in M}\left\{\left(\sum_{i=1}^n \int_{S^i}\right)dz\,\Xi(*)\bm{d}(z,z')^{2\alpha-d} \right\}.
    \end{align*}
 Here we have used \eqref{int Gaussian M} for the last step.
   To proceed, we apply lemma \ref{behavior of $F$ in sausages} to $\Xi(*)$ and estimate in \eqref{est: int Gaussian and convariance} in order to obtain for each $1\leq i\leq n$, \begin{align*}
        \int_{S^i}dz \Xi(*)\bm{d}(z,z')^{2\alpha-d}&\leq \int_{S^i}dz \left(2\pi\frac{s(t-s)}{t}\right)^{-\frac{d}{2}}\exp\left\{\frac{-\bm{d}(z,\gamma_i(s/t))^2}{2\frac{s(t-s)}{t}}\right\}\bm{d}(z,z')^{2\alpha-d}\\
        &\leq \int_{M}dz (\pi s)^{-\frac{d}{2}}e^{-\frac{\bm{d}(z,\gamma_i(s/t))^2}{2s}}\bm{d}(z,z')^{2\alpha-d}\\
        &\leq C_M(1+s^{\frac{2\alpha-d}{2}}).
    \end{align*}
We thus have,
 \begin{align}\label{est: short time key term 2-1}
         \sup_{t\geq 2s}\sup_{x_0,x,x_0',x'\in M}\iint_{M^2}dzdz' \Xi(*)(\pi t)^{-\frac{d}{2}}e^{-\frac{\bm{d}(x_0',z')^2}{2t}}\bm{d}(z,z')^{2\alpha-d}\leq C_M(1+s^{\frac{2\alpha-d}{2}}).
\end{align}
    
   % Since $\gamma_i(s/t)$ is a point in $M$, we observe that the above is the same as the space integral in \eqref{weird cross term}, from which we simply apply the calculation done in the proof of lemma \ref{lemma:Easy terms} to finish treating this term.   

%%%%%%%%%%%%%%%%%%%%%%%%%%%%%%%%%%%%%%%%%%%%

    For the Gaussian term with $s$ in \eqref{short time key term (2)}, Remark \ref{sausages cover $M$} gives us \begin{align*}
        & \iint_{M^2}dzdz' \Xi(*)(2\pi s)^{-\frac{d}{2}}e^{-\frac{\bm{d}(z',x')^2}{2s}}\bm{d}(z,z')^{2\alpha-d}\\
        \leq &  \left(\sum_{i=1}^n \int_{S^i}\right)dz\left(\int_{B(x',\delta)}+\int_{B(x',\delta)^c}\right)dz'\Xi(*)(2\pi s)^{-\frac{d}{2}}e^{-\frac{\bm{d}(z',x')^2}{2s}}\bm{d}(z,z')^{2\alpha-d}\\
        =& \left(\sum_{i=1}^n \iint_{S^i\times B(x',\delta)}+\sum_{i=1}^n\iint_{S^i\times B(x',\delta)^c}\right)dzdz'\Xi(*)(2\pi s)^{-\frac{d}{2}}e^{-\frac{\bm{d}(z',x')^2}{2s}}\bm{d}(z,z')^{2\alpha-d}.
    \end{align*}
    Using Lemma \ref{behavior of $F$ in sausages} for $\Xi(*)$, it is clear that one can treat the integrals in the second summation similarly to $J_2(s)$ in the proof of Theorem \ref{theorem:large time L1 bound}, which leads to an upper bound by  ${C_M}{s^{-d/2}}e^{-\frac{\delta^2}{s}}$.
    For the first summation, applying again Lemma \ref{behavior of $F$ in sausages} to $\Xi(*)$ gives
    \begin{align}\label{proof of th12 (ii)1}
        &\iint_{S^i\times B(x',\delta)} \Xi(*)(2\pi s)^{-\frac{d}{2}}e^{-\frac{\bm{d}(z',x')^2}{2s}}\bm{d}(z,z')^{2\alpha-d}dzdz'\nonumber\\
        \leq & \left(\iint_{B^i(s)\times B(x',\delta)}+\iint_{C^i(s)\times B(x',\delta)}\right) \Xi(*)(2\pi s)^{-\frac{d}{2}}e^{-\frac{\bm{d}(z',x')^2}{2s}}\bm{d}(z,z')^{2\alpha-d}dzdz'\nonumber\\
        \leq & \iint_{B(\gamma_i(s/t),\delta)\times B(x',\delta)} (\pi s)^{-d}e^{-\frac{d(z,\gamma_i(s/t))^2}{2s}}e^{-\frac{d(z'x')^2}{2s}}\bm{d}(z,z')^{2\alpha-d}dzdz'\nonumber\\
        &+\iint_{M\times B(x',\delta)} (\pi s)^{-d}e^{-\frac{\delta^2}{2s}}e^{-\frac{d(z'x')^2}{2s}}\bm{d}(z,z')^{2\alpha-d}dzdz'.
    \end{align}
    Note that have chosen $\delta=i_M/8$, the first  integral (and second, respectively) on the right-hand side of \eqref{proof of th12 (ii)1} can be treated in the same as $J_1(s)$ (and $J_2(s)$, respectively) in the proof of Theorem \ref{theorem:large time L1 bound}. We therefore conclude that 
    \begin{align*}
\iint_{S^i\times B(x',\delta)} \Xi(*)(2\pi s)^{-\frac{d}{2}}e^{-\frac{\bm{d}(z',x')^2}{2s}}\bm{d}(z,z')^{2\alpha-d}dzdz'\leq C_M(1+s^{\frac{2\alpha-d}{2}}).
    \end{align*}
    Hence 
     \begin{align}\label{est: short time key term 2-2}
         \sup_{t\geq 2s}\sup_{x_0,x,x_0',x'\in M}\iint_{M^2}dzdz' \Xi(*)(2\pi s)^{-\frac{d}{2}}e^{-\frac{\bm{d}(x_0',z')^2}{2s}}\bm{d}(z,z')^{2\alpha-d}\leq C_M(1+s^{\frac{2\alpha-d}{2}}).
\end{align}

Collecting \eqref{est: short time key term 2-1} and \eqref{est: short time key term 2-2}, we have
\begin{align}\label{est: short time key term 2}
    \sup_{t\geq 2s}\sup_{x_0,x,x_0',x'\in M}  \iint_{M^2}dzdz' \Xi(*)f(*')\bm{d}(z,z')^{2\alpha-d}\leq C_M(1+s^{\frac{2\alpha-d}{2}}).
\end{align}

    Finally, for integral (iii), Remark \ref{sausages cover $M$} implies
    $$\iint_{M\times M}\leq \sum_{i=1}^n\sum_{j=1}^m \iint_{S^i\times {S^{j}}'}=\sum_{i=1}^n\sum_{j=1}^m \iint_{B^i(s)\times B^j(s)'}+\iint_{B^i(s)\times C^j(s)'}+\iint_{C^i(s)\times B^j(s)'}+\iint_{C^i(s)\times C^j(s)'}.$$
    For each summand, we first apply lemma \ref{behavior of $F$ in sausages} to $\Xi(*)$ and $\Xi(*')$, then each term can be estimated similarly to $J_1$(s), $J_2(s)$, $J_3(s)$ and $J_4(s)$ in proof of Theorem \ref{theorem:large time L1 bound}. We thus have
    \begin{align}\label{est: short time key 3}
 \sup_{t\geq 2s}\sup_{x_0,x,x_0',x'\in M}\iint_{M^2}dzdz' \Xi(*)\Xi(*')\bm{d}(z,z')^{2\alpha-d}\leq C_M(1+s^{\frac{2\alpha-d}{2}}).
    \end{align}
    Combining \eqref{small time first term}, \eqref{est: short time key term 2} and \eqref{est: short time key 3}, the proof is thus completed.

\subsection{Upper bound for $\mathcal{L}_n$}
Combining Theorem \ref{theorem:large time L1 bound} and Theorem \ref{theorem:short time L1 bound}, we have the following. Define for all $s>0$
\begin{align}\label{kappa}
k(s):=k_1(s)+k_2(s).
\end{align}
\begin{theorem}\label{theorem:combined L1 bound}
    Suppose $M$ has non-positive sectional curvature. Then for any $t\in(0,\infty)$, $x_0,x,x_0',x'\in M$, $$\mathcal{L}_1(t,x_0,x,x_0',x')\leq (C_L+C_S)G_t(x_0,x)G_t(x_0',x')\left(\int_0^t k(s)ds\right).$$
\end{theorem}
Observe that $C_L+C_S$ does not depend on space arguments, which is essential for inductively bounding $\mathcal{L}_n$. For the same purpose, we will need the following elementary lemma.
\begin{lemma}\label{lemma:hn non-decreasing}
    Define inductively $\set{h_n(t)}_{n\geq 1}$ by $$h_1(t)=\int_0^t k(s)ds,\quad\mathrm{and}\quad h_{n}(t)=\int_0^t h_{n-1}(t-s)k(s)ds,n\geq 2.$$
    Then $h_n$ is non-decreasing for all $n\geq1$.
\end{lemma}
\begin{proof}
    We proceed by induction. 
    The case $n=1$ is true by non-negativity of $k(t)$. 
    Now suppose it holds up to $n$. 
    We then have \begin{align*}
        h_{n+1}(t+\varepsilon)=&\int_0^{t+\varepsilon} h_n(t+\varepsilon-s)k(s)ds\\
        \geq&\int_0^t h_n(t+\varepsilon-s)k(s)ds\\
        \geq&\int_0^t h_n(t-s)k(s)ds=h_n(t).
    \end{align*}
\end{proof}

The following theorem gives the desired estimate for $\mathcal{L}_n.$
\begin{theorem}\label{thm: Ln bound}
    There exits $C>0$ depending only on $\alpha$ and $M$ such that for all $t>0$ and $x_0,x,x_0',x'\in M$, we have \begin{align}\label{bound: L_n}\mathcal{L}_n(t,x_0,x,x_0',x')\leq 2^n C^n G_t(x_0,x)G_t(x_0',x')h_n(t).\end{align}
\end{theorem}
\begin{proof}
    We again proceed by induction, where the case $n=1$ is the content of Theorem~\ref{theorem:combined L1 bound}. 
    Now suppose it holds up to $n-1$. We thus have \begin{align*}
        \mathcal{L}_n=&\int_0^t ds\iint_{M^2} dzdz' P_{t-s}(z,x)P_{t-s}(z',x') \mathcal{L}_{n-1}(s, x_0,z,x_0',z') \bm{G}_{\alpha,\rho}(z,z')\\
        \leq& (2C)^{n-1}G_t(x_0,x)G_t(x_0',x')\int_0^t ds h_{n-1}(s)\iint_{M^2}dzdz'G_{t,x_0,x}(s,z)G_{t,x_0',x'}(s,z')\bm{G}_{\alpha,\rho}(z,z')
    \end{align*}
    By lemma \ref{lemma:hn non-decreasing} and symmetry of the roles of $s$ and $t-s$ in $G_{t,x_0,x}(s,z)$, we need only to upper bound $$\int_{\frac{t}{2}}^t ds h_{n-1}(s)\iint_{M^2}dzdz' G_{t,x_0,x}(s,z)G_{t,x_0',x'}(s,z')\bm{G}_{\alpha,\rho}(z,z')$$ because $\int_0^{\frac{t}{2}}$ can be treated likewise.
    A change of variables $s=t-s$ shows that the above equals $$\int_0^{\frac{t}{2}}ds h_{n-1}(t-s)\iint_{M^2}dzdz' \bm{G}_{\alpha,\rho}(z,z') G_{t,x,x_0}(s,z)G_{t,x',x_0'}(s,z').$$
    The space integral is handled in large time the same as in Section 4 and small time the same as in Section 5, giving us $$\int_0^{\frac{t}{2}}ds h_{n-1}(t-s)\iint_{M^2}dzdz'\bm{G}_{\alpha,\rho}(z,z') G_{t,x,x_0}(s,z)G_{t,x',x_0'}(s,z')\leq C\int_0^{\frac{t}{2}}h_{n-1}(t-s)k(s)ds.$$
    Adding with the part which starts with $\int_0^{\frac{t}{2}}$ gives \eqref{bound: L_n}. Now recall the definition of $\mathcal{K}_\beta$ in \eqref{K_beta}, the upper bound in \eqref{boudn: K_beta} is a direct consequence of \eqref{bound: L_n} and the definition of $H_\lambda$.
\end{proof}

\section{Well-Posedness and Moment Upper Bound}
We are now ready to prove the well-posedness and moments upper bounds for equation \eqref{PAM}. Recall the iteration procedure outlined at the beginning of Section 3. In particular, equation \eqref{iteration corelation} implies that the existence of an $L^2$-solution to \eqref{PAM} relies on the convergence of the series,
\begin{equation*}
    \mathcal{K}_\beta(t,x,z,x',z')=\sum_{n=0}^\infty \beta^{2n}\mathcal{L}_n(t,x,z,x',z').
\end{equation*}
Now that $\mathcal{L}_n$ is controlled by $h_n$ thanks to Theorem \ref{thm: Ln bound}, we set for any $\lambda>0$,  
$$H_\lambda(t):=\sum_{n=0}^\infty \lambda^{2n}h_n(t).$$

\begin{corollary}\label{cor:K beta bound}
    For any $t>0$ and $x,x_0,x_0',x'\in M$, we have \begin{align}\label{boudn: K_beta}\mathcal{K}_\beta(t,x_0,x,x_0',x')\leq G_t(x_0,x)G_t(x_0',x')H_{2\beta^2 C}(t).\end{align}
\end{corollary}
\begin{proof}
    This follows trivially from the definition of $\mathcal{K}_\beta$ and Theorem \ref{thm: Ln bound}.
\end{proof}

The following result for $H$ is needed to obtain exponential (in time) moment bounds for the solution $u$.
\begin{lemma}\label{lemma:behavior of H}
    Let $\alpha>\frac{d-2}{2}$, $\lambda>0$.  There exist constants $C, \theta>0$ depending on $\alpha,\lambda$ such that for all $t>0$, $$H_{\lambda}(t)\leq C e^{\theta t}.$$
\end{lemma}
\begin{proof}
    The proof is taken from \cite[Lemma 2.5]{ChenKim19}.  We have for all $\gamma>0$, 
    \begin{align}\label{lem 26 midstep 1}\int_0^\infty e^{-\gamma t}h_n(t)dt=\frac{1}{\gamma}\left(\int_0^\infty e^{-\gamma t}k(t)dt\right)^n.
    \end{align}
    Theorem \ref{theorem:large time L1 bound} and Theorem \ref{theorem:short time L1 bound} implies 
    $$k(t)\leq C_M(1+t^{\frac{2\alpha-d}{2}}).$$
    Together with our assumption on $\alpha$, the integral on the right-hand side of \eqref{lem 26 midstep 1} is finite and decreases to 0 as $\gamma\uparrow \infty$. Clearly we can select $\theta:=\inf\set{\gamma >0: \int_{\R_+} e^{-\gamma t}k(t)dt<\frac{1}{\lambda^2}}$. This would give us for all $\gamma>\theta$ $$\int_{\R_+}H_\lambda(t)e^{-\gamma t}dt=\int_{\R_+}\sum_{n=0}^{\infty}\lambda^{2n}h_n(t)e^{-\gamma t}dt\leq \frac{1}{\gamma}\sum_{n=0}^{\infty}\lambda^{2n}\left(\int_{\R_+}e^{-\gamma t}k(t)dt\right)^n<\infty.$$
     This together with the fact that $H_\lambda$ is non-decreasing (since $h_n's$ are) implies the desired bound for $H_\lambda$.
\end{proof}
%We now introduce some terminology defined analogously as section 5 of \cite{COV23} to rigorously define the mild solution of our SPDE.\\
We now fully state the first main result of the paper. Let $\mathcal{B}$ be the Borel $\sigma-$algebra of $M$. For $A\in \mathcal{B}$, $t\geq 0$, define $W_t(A):=W(\1_{[0,t]}(s)\1_{A}(x)).$  Define the filtration $(\mathcal{F}_t)_{t\geq 0}$ by $$\mathcal{F}_t:=\sigma(W_s(A):0\leq s\leq t, A\in \mathcal{B})\vee \mathcal{N},$$ where $\mathcal{N}$ is the collection of $\P-$null sets of $\mathcal{F}$. 
\begin{comment}
    Define the \ti{predictable $\sigma$-algebra} $\mathcal{P}$ as the $\sigma$-algebra generated by the subsets of $\R_+\times M\times \Omega$ which contain all sets of forms $\set{0}\times F_0$ and  $]s,t[\times A\times F$, where $F_0\in \mathcal{F}_0$, $F\in \mathcal{F}_s,0\leq s<t$ and $A$ is some geodesic ball in $M$. Sets in $\mathcal{P}$ are called \ti{predictable sets}. A random field $X:\Omega\times \R_+\times M\ra \R$ is \ti{predictable} if $X$ is $\mathcal{P}-$measurable.\\
For $p\in [2,+\infty[$, denote by $\mathcal{P}_p$ the class of predictable and $L^2_{\bm{G}_{\alpha,\rho}}(R_+\times M^2;L^p(\Omega))$ random fields. 
More precisely, $f\in \mathcal{P}_p$ iff $f$ is predictable and $$\norm{f}^2_{W,p}:=\iiint_{\R_+\times M^2}\bm{G}_{\alpha,\rho}(x,y)\norm{f(s,x)f(s,y)}_{\frac{p}{2}}dxdyds<+\infty,$$ where $\norm{\cdot}_p$ is the $L^p(\Omega)$-norm. Clearly for $2\leq p\leq q<+\infty,$ we have $\mathcal{P}_2\supseteq \mathcal{P}_p\supseteq\mathcal{P}_q$.\\
\end{comment}
\begin{definition}\label{Adapted mild solution}
    A random field $\set{u(t,x)}_{t\geq 0,x\in M}$ is an It\^{o} mild solution to the Cauchy problem if all the following holds.
    \begin{enumerate}[label=(\roman*)]
        \item Every $u(t,x)$ is $\mathcal{F}_t-$measurable.
        \item $u(t,x)$ is jointly measurable with respect to $\mathrm{B}((0,\infty)\times M)\otimes \mathcal{F}$
        \item For all $(t,x)\in (0,\infty)\times M$, we have $$\E \left[\int_0^tds\iint_{M^2}dzdz' \bm{G}_{\alpha,\rho}(z,z')P_{t-s}(x,z)u(s,z)P_{t-s}(x,z')u(s,z')\right]< \infty$$
        \item $u$ satisfies \eqref{pam:mild}.%$u(t,x)=J_0(t,x)+\beta I(t,x)$. {\color{red}$I(t,x)$ is not introduced yet.}
    \end{enumerate}
\end{definition}

\begin{comment}
    \begin{theorem}
    Suppose $t>0$, $p\in [2,\infty[$, and the random field $$X=\set{X(s,y):(s,y)\in ]0,t[\times M}$$
    has the following properties: \begin{enumerate}
        \item $X$ is adapted, ie. for all $(s,y)\in ]0,t[\times M$, $X(s,y)$ is $\mathcal{F}_s-$measurable;
        \item $X$ is jointly measurable with respect to $\mathcal{B}(\R_+\times M)\times \mathcal{F}$;
        \item $\norm{X}_{W,p}<+\infty$.
    \end{enumerate}
    Then $X(\cdot,\circ)\1_{]0,t[}(\cdot)$ belongs to $\mathcal{P}_p$.
\end{theorem}
This theorem is formulated like the result in [reference], and the proof is almost exactly the same except with the modifications like those in [reference].\\
\end{comment}

\begin{theorem}\label{main result1}
    For any $\alpha>\frac{d-2}{2}$ and finite measure $\mu$ on $M$, the Cauchy problem \eqref{PAM} has a random field solution $\set{u(t,x)}_{t>0,x\in M}$ which is $L^p(\Omega)$ continuous for $p\geq 2$ and satisfies the two-point correlation formula $$\E[u(t,x)u(t,x')]=J_1(t,x,x')+ \beta^2\iint_{M^2} \mu(dz)\mu(dz') K_\beta(t,z,x,z',x').$$
    Also the following moment bound holds, where $C=C_L+C_S$, $C',\theta>0$ depending on $\alpha,\beta,C$ and $p$ : $$\E[|u(t,x)|^p]^{\frac{1}{p}}\leq \sqrt{2}J_0(t,x)\left(H_{4\beta C\sqrt{p}}(t)\right)^{\frac{1}{2}}\leq C'J_0(t,x)e^{\theta t}.$$
\end{theorem}
\begin{proof}
    The six-step Picard iteration scheme used in \cite{ChenThesis,ChenDalangAOP} with the modifications presented in \cite{ChenKim19} is usable here to obtain $L^2(\Omega)$ continuity and the correlation formula.
    The same proof as Theorem~1.3 in \cite{COV23} is possible by the above estimates for the first inequality in the $p$-th moment bound. The exponential bound for the $p$-th moment is due to Lemma \ref{lemma:behavior of H}.
\end{proof}

\section{Lower Bound assuming bounded initial condition}
We close our discussion by presenting an exponential lower bound for the second moment of the solution. It is proved under an extra condition that the initial data is given by a bounded measurable function under which one has the Feynman-Kac representation for the second moment of the solution to the parabolic Anderson model.\\

First recall the following spectral decomposition of the heat kernel,
$$P_t(x,y)=\sum_{n=0}^{\infty}e^{-\lambda_n t}\phi_n(x)\phi_n(y),$$
where $\set{\lambda_n}_{n=1}^\infty,0=\lambda_0<\lambda_1\leq \lambda_2...$ are the eigenvalues of $\Laplace_M$ and $\set{\phi_n}_{n=1}^\infty$ the corresponding orthonormal eigenfunctions. The definition of $\bm{G}_\alpha$ in \eqref{def: G_alpha etc} then gives $$\bm{G}_\alpha(x,y)=\sum_{n=1}^{\infty}\lambda_n^{-\alpha}\phi_n(x)\phi_n(y).$$

%\begin{proof}
 %   For $x\in M$, we look to determine the functions $\set{a_n(x)}_{n=1}^{+\infty}$ such that $$G_{\alpha}(x,y)=\sum_{n=1}^{+\infty}a_n(x)\phi_n(y).$$
 %   Spectral theory then tells us $a_j(x)=\int_M \bm{G}_\alpha(x,y)\phi_n(y)dy$, from which an application of Fubini and the orthonormality of the eigenfunctions give the result.
%\end{proof}

\begin{theorem}\label{lower bound}
    Assume $\alpha>\frac{d-2}{2}$ and  $\mu(dx)=f(x)dx$, where $f:M\ra \R$ is bounded and $\inf_{x\in M}f(x)\geq \varepsilon>0$. Suppose in addition $\rho>0$. Then there exists a positive constant $c$ such that, $$\E[u(t,x)^2]\geq \varepsilon^2 e^{c\, t},\quad\mathrm{for\ all}\ t>0.$$ 
\end{theorem}
\begin{proof}
    When $f$ is bounded, standard approximation argument gives the Feynman-Kac formula for the second moment (see, e.g., \cite{HN09,HHNT15}) $$\E[u(t,x)^2]=\E_x \left[f(B_s)f(B'_s)\exp\left\{\beta^2\int_0^t \bm{G}_{\alpha,\rho}(B_s,B'_s)ds\right\}\right],$$ where $B,B'$ are two independent Brownian motions on $M$ starting at $x$.
    Under the assumption $\inf_{x\in M}f(x)\geq \varepsilon>0$, the second moment is bounded below by \begin{align}\label{E: lower bound moment 1}\E[u(t,x)^2]\geq \varepsilon^2\E_x \left[\exp\left\{\beta^2\int_0^t \bm{G}_{\alpha,\rho}(B_s,B'_s)ds\right\}\right]\geq \varepsilon^2 \exp\left\{\beta^2\int_0^t \E_x \bm{G}_{\alpha,\rho}(B_s,B'_s)ds\right\},\end{align}
    where the second inequality follows from an application of Jensen's inequality. Recall the definition of $\bm{G}_{\alpha,\rho}$ in \eqref{def: G_alpha etc}, the exponent on the right-hand side of \eqref{E: lower bound moment 1} equal to 
    \begin{align*}\beta^2\left(\frac{\rho t}{m_0}+\E_x\left[\int_0^t \bm{G}_{\alpha}(B_s,B'_s)ds\right]\right).\end{align*}
   In order to compute the expectation above, note that for each $n\geq1$, $\phi_n$ is the eigenfunction of the Laplacian corresponding to eigenvalue $\lambda_n$, hence
   $$\E_x[\phi_n(B_s)]=\E_x[\phi_n(B'_s)]=\phi_n(x)e^{-\lambda_ns}.$$
We therefore have, as $t\uparrow\infty$,
    \begin{align*}
        \E_x\left[\int_0^t \bm{G}_{\alpha}(B_s,B'_s)ds\right]=&\int_0^t \E_x[\bm{G}_{\alpha}(B_s,B'_s)]ds=\int_0^t \sum_{n=1}^{+\infty}\lambda_n^{-\alpha}\E_x[\phi_n(B_s)\phi_n(B'_s)] ds\\
       =&\int_0^t \sum_{n=1}^{+\infty}\lambda_n^{-\alpha}\E_x[\phi_n(B_s)]^2 ds= \int_0^t \sum_{n=1}^{+\infty}\lambda_n^{-\alpha}e^{-2\lambda_ns}\phi_n(x)^2 ds\\
        =& \sum_{n=1}^{+\infty} \frac{1-e^{-2\lambda_nt}}{2\lambda_n^{\alpha+1}}\phi_n(x)^2\quad\uparrow\quad%\xrightarrow{t\uparrow +\infty} 
        \sum_{n=1}^{+\infty}\frac{\phi_n(x)^2}{2\lambda_n^{\alpha+1}}=\frac{1}{2}\bm{G}_{\alpha+1}(x,x).%\sim \sum_{n=1}^{+\infty}\frac{\phi_n(x)^2}{2[C(d,m_0)n]^{\frac{2\alpha+2}{d}}}.
    \end{align*}
    Note that the assumption on $\alpha$ for the well-posedness of equation \eqref{pam:mild} implies $\alpha+1>\frac{d}{2}$, which shows that $\bm{G}_{\alpha+1}(x,x)$ is finite thanks to Proposition \ref{Prop: G_alpha}. Hence, the exponent on the right-hand side of \eqref{E: lower bound moment 1} is of order
    $$\frac{\beta^2\rho}{m_0}t+\frac{\beta^2}{2}\bm{G}_{\alpha+1}(x,x),\quad\mathrm{as}\ t\uparrow\infty.$$
    The proof is thus completed. %This completes the proof. {\color{blue}I think the exponent should actually be approximately $\frac{\rho}{m_0}t+G_{\alpha+1}(x,x)$ in the limit?}
\end{proof}

The exponential lower bound stated in the above theorem is the result of the compactness of $M$. Indeed, a Brownian motion on a compact manifold is ergodic and hence the time average converges to the space avergage:
\begin{align*}
\frac{1}{t}\int_0^t \bm{G}_{\alpha}(B_s,B'_s)ds\ \to\  \frac{1}{m_0^2}\int_{M\times M}\bm{G}_\alpha(x,x')dxdx'=0.
\end{align*}
This is the main intuition that leads to the proof. We believe that the assumption on the initial data is only a technical assumption; we expect that the exponential lower bound still holds for rough initial data. 

\begin{remark}
    Using the fact that the $p$-th moment is lower bounded by the second moment for $p\geq2$, ones also obtains an exponential lower bound for the $p$-th moment, which matches the upper bound proved in the previous sections.
\end{remark}

\begin{remark}
  The argument for the lower bound relies on the specific construction of the covariance function, which allows for an explicit analysis of the action of the heat semigroup on $\bm{G}_\alpha$. It is not clear how this approach extends to more general noises. In contrast, the upper bound depends primarily on Proposition \ref{Prop: G_alpha} from the covariance structure of the noise, and therefore continues to hold for a broader class of noises with similarly behaved covariance functions.
\end{remark}

\printbibliography

%\bibliographystyle{alpha}
%\bibliography{citations}

\end{document}